\documentclass[reqno]{amsart}
\usepackage{amsmath,amsbsy,amsfonts}
\usepackage{color}

\usepackage[utf8]{inputenc}
\usepackage{graphicx}
\usepackage{amsmath,latexsym,amssymb}
\newtheorem{theorem}{Theorem}
\newtheorem{proposition}{Proposition}[section]
\newtheorem{corollary}[proposition]{Corollary}
\newtheorem{definition}[proposition]{Definition}
\newtheorem{lemma}[proposition]{Lemma}
\newtheorem{remark}[proposition]{Remark}

\newcommand{\diver}{\operatorname{div}}
\newcommand{\dd}{\mathrm{d}}

\numberwithin{equation}{section}

\begin{document}
	\title[Pseudo-relativistic Hartree equation]{Remarks about a generalized pseudo-relativistic Hartree equation}
	
	\author{H. Bueno}
	\address{H. Bueno and G. A. Pereira- Departmento de Matem\'atica, Universidade Federal de  Minas Gerais, 31270-901 - Belo Horizonte - MG, Brazil}
	\email{hamilton.pb@gmail.com and  gilbertoapereira@yahoo.com.br }
	\author{O. H. Miyagaki}
	\address{O. H. Miyagaki - Departmento de Matem\'atica, Universidade Federal de Juiz de Fora, 36036-330 - Juiz de Fora - MG, Brazil}
	\email{ohmiyagaki@gmail.com}
	\author{G. A. Pereira}

	\thanks{H. Bueno is the corresponding author; O. H. Miyagaki has received research grants from CNPq/Brazil  304015/2014-8 and INCTMAT/CNPQ/Brazil; G. A. Pereira received research grants by PNPD/CAPES/Brazil}
	\date{}

\begin{abstract}
With appropriate hypotheses on the nonlinearity $f$, we prove the existence of a ground state solution $u$ for the problem
\[(-\Delta+m^2)^\sigma u+Vu=\left(W*F(u)\right)f(u)\ \ \text{in }\ \mathbb{R}^{N},\]
where $0<\sigma<1$, $V$ is a bounded continuous potential and $F$ the primitive of $f$. We also show results about the regularity of any solution of this problem.
\end{abstract}
\maketitle

\section{Introduction}\label{intro}

The study of the generalized pseudo-relativistic Hartree equation
\begin{equation}\label{sqroot}\sqrt{-\Delta+m^2}\, u+Vu=\left(W*F(u)\right)f(u)\ \ \text{in }\ \mathbb{R}^{N},
\end{equation}
where $F(t)=\int_0^t f(s)\dd s$, was carried out in \cite{BBMP} with adequate hypotheses in the case $N\geq 2$. The study of \eqref{sqroot} is based on the pioneering work of  Coti Zelati and Nolasco \cite{ZelatiNolasco} and was generalized by Cingolani and Secchi \cite{Cingolani}. 

The purpose of this paper is to consider the same equation \eqref{sqroot}, substituting the operator $\sqrt{-\Delta+m^2}\, u$ by $(-\Delta+m^2)^\sigma u$, where $0<\sigma<1$. Namely, we consider the equation
\begin{equation}\label{original}(-\Delta+m^2)^\sigma u+Vu=\left(W*F(u)\right)f(u)\ \ \text{in }\ \mathbb{R}^{N},\end{equation}
supposing that the potential $V\colon\mathbb{R}^N\to \mathbb{R}$ is a continuous, (possibly) sign-changing \emph{bounded} function satisfying
\begin{enumerate}
	\item [($V1$)] $V(y)+V_0\geq 0$ for every $y\in\mathbb{R}^{N}$ and some constant $V_0<\min\{1,m^2\}\mathcal{K}(\Phi_\sigma)$, where the constant $\mathcal{K}(\Phi_\sigma)>0$ will be defined later;
	\item [($V2$)] $V_\infty=\displaystyle\lim_{|y|\to\infty}V(y)>0$;
	\item [($V3$)] $V(y)\leq V_\infty$ for all $y\in\mathbb{R}^N$, $V(y)\neq V_\infty$.
\end{enumerate}

Hypotheses like ($V_2$) appear since the work of P.-L Lions \cite{CC1} in the context of positive or sign-changing potentials, see e.g. \cite{Rabinowitz}. 

We assume that the \emph{radial} function $W$ satisfy
\begin{enumerate}
		\item [($W_h$)] $0\leq W=W_1+W_2\in L^r(\mathbb{R}^{N})+L^\infty(\mathbb{R}^{N})$, with $r>\frac{N}{N(2-\theta)+2\sigma\theta}$.
\end{enumerate}

Condition ($W_h$) goes back to Cingolani and Secchi \cite{Cingolani}. In that paper, \eqref{sqroot} is considered with the homogeneous term $(W*|u|^\theta)|u|^{\theta-2}u$ instead of $\left(W*F(u)\right)f(u)$; the possibly sign-changing function $V$ is required to be continuous, with $V(y)\leq V_\infty-e^{k|x|}$ for an appropriate value of the constant $k$ and all $|x|>R$.

We accept that the nonlinearity $f$  is a $C^1$ function that satisfy
\begin{enumerate}
\item [($f1$)] $\displaystyle\lim_{t\to 0}\frac{|f(t)|}{t}=0$;
\item [($f2$)] $\displaystyle\lim_{t\to \infty}\frac{f(t)}{t^{\theta-1}}=0$ for some $\theta$ such that $\max\{2,N/(N-2\sigma)\}<\theta<2^{*}_\sigma=\frac{2N}{N-2\sigma}$;
\item [($f3$)] $\displaystyle\frac{f(t)}{t}$ is increasing for all $t>0$.
\end{enumerate}
Considering $\sigma=1/2$, these hypotheses were also assumed by the authors of \cite{BBMP}.

Condition $(f3)$ implies the Ambrosetti-Rabinowitz condition $2F(t)\leq tf(t)$ if $t>0$, which is satisfied, for example, by
\[f(t)=t\ln(1+t),\]
a function that does not satisfy $\theta F(t)\leq tf(t)$ for any $\theta>2$. 

Concerning the applications of equation \eqref{original}, we recall that fractional Laplacian operators are the infinitesimal generators of Lévy stable diffusion processes. They have application in several areas such as anomalous diffusion of plasmas, probability, finances and populations dynamics, see \cite{Applebaum}. The special case $\sigma=1/2$ conducts to equation \eqref{sqroot}, a pseudo-relativistic Hartree equation, see V. Moroz and  J. Van Schaftingen \cite{Moroz1,Moroz2} and references therein.
 
Our approach studying equations \eqref{original} rests on the seminal papers by Cabré and Solà-Morales \cite{Cabre} and Caffarelli and Silvestre \cite{Caffarelli}. Therefore, we consider the Dirichlet-to-Neumann operator, that is, the extension problem naturally related to \eqref{original} for the operator $(-\Delta+m^2)^\sigma u$. 

We state a general result about the extension problem, just changing the notation:

\textbf{Theorem} (Stinga-Torrea \cite{StingaTorrea1})\textit{ Let $h\in Dom(L^\sigma)$ and $\Omega$ be an open subset of $\mathbb{R}^N$. A solution of the extension problem
	\[\left\{\begin{array}{ll}
	-L_y u+\frac{1-2\sigma}{x}u_x+u_{xx}=0 &\text{in }\ (0,\infty)\times \Omega\\
	u(0,y)=h(y) &\text{on }\ \{x=0\}\times\Omega
	\end{array}\right.\]
	is given by
	\[u(x,y)=\frac{1}{\Gamma(\sigma)}\int_{0}^\infty e^{-tL}(L^\sigma h)(y)e^{-\frac{x^2}{4t}}\frac{\dd t}{t^{1-\sigma}}\]
	and satisfies
	\[\lim_{x\to 0}\frac{x^{1-2\sigma}}{2\sigma}u_x(x,y)=\frac{\Gamma(-\sigma)}{4^\sigma\Gamma(\sigma)}(L^\sigma h)(y).\]}\vspace*{.2cm}

In our case, the extension problem produces, for $(x,y)\in (0,\infty)\times\mathbb{R}^N=\mathbb{R}^{N+1}_+$,
\begin{equation}\left\{\begin{array}{ll}
\displaystyle\Delta_y u+\frac{1-2\sigma}{x}u_x+u_{xx}-m^2u=0 &\text{in }\ \mathbb{R}^{N+1}_+ \\ \\
\displaystyle\lim_{x\to 0^+}\left(-x^{1-2\sigma}\frac{\partial u}{\partial x}\right)= -V(y) u+ \left[W*F(u)\right]f(u), &\text{in } \left\{ 0\right\} \times \mathbb{R}^N \simeq\mathbb{R}^N,
\end{array}\right.\tag{$P$}\label{P}
\end{equation}

However, our approach is not based on the action of the heat semigroup  $e^{tL}$ generated by the operator $L$ acting on $L^\sigma h$, for $h\in Dom(L^\sigma)$. As in Brändle, Colorado, de Pablo and Sánchez \cite{Brandle}, the Fourier transform will be our main technique.

We summarize our main existence result, which will be proved in Section \ref{GS}:
\begin{theorem}\label{t1}
	Suppose that conditions \textup{($f1$)-($f3$)}, \textup{($V1$)-($V3$)} and \textup{($W_h$)} are valid. Then, problem \eqref{P} has a \emph{positive} ground state solution $u\in H^1(\mathbb{R}^{N+1}_+,x^{1-2\sigma})$. 
\end{theorem}\vspace*{.2cm}

The space $H^1(\mathbb{R}^{N+1}_+,x^{1-2\sigma})$ will be defined in Section \ref{FS}. 

Theorem \ref{t1} will be obtained by comparing problem \eqref{original} with its asymptotic version, where $V(y)$ is substituted by $V_\infty$. The ground state solution for the asymptotic problem yields the minimal energy solution of \eqref{original}, by applying the so-called splitting lemma of Struwe \cite{Struwe}.

Once obtained a solution of \eqref{original}, the natural step is to consider its regularity. We will prove that any weak solution $v$ of \eqref{original} belongs to $L^\infty(\mathbb{R}^{N+1}_+)$. Then, applying a classical result of Fabes, Kenig and Serapioni \cite{Fabes} (see also \cite{CabreSire}) we conclude that $v\in C^\alpha(\mathbb{R}^{N+1}_+)$.
\begin{theorem}\label{t2}
	Any solution $v$ of problem \eqref{P} satisfies
	\[v\in L^\infty(\mathbb{R}^{N+1}_+)\cap C^\alpha(\mathbb{R}^{N+1}_+).\]
\end{theorem}

It should be mentioned that this regularity of $v$ does not depend on hypotheses $(f3)$, but only on $(f1)$ and $(f2)$. 

The path that leads to Theorem \ref{t2} is arduous: cut-off and a bootstrap argument shows that $v(0,\cdot)\in L^p(\mathbb{R}^N)$ for all $p\in [2,\infty)$. The iteration process of Moser proves that $v\in L^\infty(\mathbb{R}^{N+1}_+)$. And the classical result of Fabes, Kenig and Serapioni \cite[Theorems 2.3.12 and 2.3.15]{Fabes} conduces to the desired result.

Finally, we prove
\begin{theorem}\label{t3}Suppose that $v\in H^1(\mathbb{R}^{N+1}_+,x^{1-2\sigma})$ is a critical point of the energy functional $I$, then $v\in C^\alpha(\mathbb{R}^{N+1}_+)\cap L^\infty(\mathbb{R}^{N+1}_+)$ satisfies
	\[\sup_{y\in\mathbb{R}^{N}}|v(x,y)|\leq C|h|_2x^{(2\sigma-1)/2}e^{-mx}\]	
	and therefore
	\[|v(x,y)|e^{\lambda x}\to 0\]
	as $x\to \infty$, for any $\lambda<m$.
\end{theorem}

The paper is organized as follows. Function spaces and immersions are treated in Section \ref{FS}.  Some preliminaries results related to the energy functional are exposed in Section \ref{Prelim}. The existence of a ground state solution for the asymptotic problem (where $V(y)$ is replaced by $V_\infty$) and the ground state solution of \eqref{original} is obtained in Section \ref{GS}, where both problems are related and solved. Theorem \ref{t2} and Theorem \ref{t3} are proved in Sections \ref{regularity} and \ref{decay}, respectively. In the Appendix we prove a Hopf-type principle that leads to the positivity of any solution of problem \eqref{P} in $\overline{\mathbb{R}^{N+1}_+}$.
\section{Function Spaces}\label{FS}

Following close \cite{Brandle}, by taking the Fourier transform in $y\in\mathbb{R}^N$ for a fixed $x>0$ of problem \eqref{original} we have
\begin{equation}\label{Fourier}\left\{\begin{array}{ll}
-(m^2+4\pi^2|\xi|^2)\hat{u}(x,\xi)+\frac{1-2\sigma}{x}\hat{u}_x+\hat{u}_{xx}=0, &\text{in }\ \mathbb{R}^{N+1}_+,\\
\hat{u}(0,\xi)=\hat{h}(\xi).\end{array}\right.\end{equation}

Therefore, the solution of \eqref{Fourier} is given by
\begin{equation}\label{transform}\hat{u}(x,\xi)=\hat{h}(\xi)\Phi_\sigma\left(\sqrt{m^2+4\pi^2|\xi|^2}\,x\right),\end{equation}
where $\Phi_\sigma$ solves (see \cite{Brandle,Capella})
\begin{equation}\label{Phi}-\Phi+\frac{1-2\sigma}{s}\Phi'+\Phi''=0,\qquad \Phi(0)=1,\quad \lim_{s\to\infty}\Phi(s)=0.
\end{equation}
We denote $c=\sqrt{m^2+4\pi^2|\xi|^2}$. The ordinary differential equation \eqref{Phi} is a Bessel equation and its solution $\Phi_\sigma$  a minimum of the functional
\[\mathcal{K}(\Phi)=\int_0^\infty \left(|\Phi(s)|^2+|\Phi'(s)|^2\right)s^{1-2\sigma}\dd s\]
and satisfies the asymptotic behavior 
\begin{equation}\label{asymp}\Phi_{\sigma}(s)\sim \left\{\begin{array}{ll}
1-c_1s^{2\sigma} &\text{when }\ s\to 0,\\
c_2s^{(2\sigma-1)/2}e^{-s} &\text{when }\ s\to \infty,\end{array}\right.\end{equation}
where the constants $c_1$ and $c_2$ depend on $\sigma$, see \cite{Brandle, Capella}. Therefore, as consequence of \eqref{asymp}, $0<\mathcal{K}(\Phi_\sigma)<\infty$.

Observe that
\begin{align*}
\int_{\mathbb{R}^{N}}|\nabla u(x,y)|^2\dd y&=\int_{\mathbb{R}^{N}}\left(|\nabla_y u(x,y)|^2+\left|\frac{\partial u}{\partial x}(x,y)\right|^2\right)\dd y\\
&=\int_{\mathbb{R}^{N}}\left(4\pi^2|\xi|^2|\hat{u}(x,\xi)|^2+\left|\frac{\partial \hat{u}}{\partial x}(x,\xi)\right|^2\right)\dd \xi.
\end{align*}
Multiplying by $x^{1-2\sigma}$ and integrating in $x$, we obtain
\begin{multline*}
\iint_{\mathbb{R}^{N+1}_+}\left(|\nabla u(x,y)|^2+m^2|u(x,y)|^2\right)x^{1-2\sigma}\dd y\dd x
\end{multline*}\vspace*{-.3cm}
\begin{align*}
&=\int_0^{\infty}\int_{\mathbb{R}^{N}}|\nabla u(x,y)|^2x^{1-2\sigma}\dd y\dd x+\int_0^\infty\int_{\mathbb{R}^{N}}m^2|u(x,y)|^2x^{1-2\sigma}\dd y\dd x\\
&=\int_0^\infty\int_{\mathbb{R}^{N}}\left(c^2|\hat{u}(x,\xi)|^2+\left|\frac{\partial \hat{u}}{\partial x}(x,\xi)\right|^2\right)x^{1-2\sigma}\dd \xi\dd x\\
&=\int_0^\infty\int_{\mathbb{R}^{N}}\left(c^2|\hat{h}(\xi)|^2|\Phi(cx)|^2+|\hat{h}(\xi)c\Phi'(cx)|^2\right)x^{1-2\sigma}\dd \xi\dd x\\
&=\int_0^\infty\int_{\mathbb{R}^{N}}c^2|\hat{h}(\xi)|^2\left(|\Phi(cx)|^2+|\Phi'(cx)|^2\right)x^{1-2\sigma}\dd \xi\dd x
\end{align*}
The change of variables $s=cx$ in the last integral yields
\begin{subequations}
\begin{multline*}
\iint_{\mathbb{R}^{N+1}_+}\left(|\nabla u(x,y)|^2+m^2|u(x,y)|^2\right)x^{1-2\sigma}\dd y\dd x
\end{multline*}
\begin{align}
&=\int_0^{\infty}\int_{\mathbb{R}^{N}}c^{2\sigma}|\hat{h}(\xi)|^2\left(|\Phi(s)|^2+|\Phi'(s)|^2\right)s^{1-2\sigma}\dd\xi\dd s\nonumber\\
&=\int_{\mathbb{R}^{N}}c^{2\sigma}|\hat{h}(\xi)|^2\dd\xi\left(\int_0^{\infty}\left(|\Phi(s)|^2+|\Phi'(s)|^2\right)s^{1-2\sigma}\dd s\right)\nonumber\\
&\geq\mathcal{K}(\Phi_\sigma)\int_{\mathbb{R}^{N}}c^{2\sigma}|\hat{h}(\xi)|^2\dd\xi
=\mathcal{K}(\Phi_\sigma)\int_{\mathbb{R}^{N}}\left(4\pi^2|\xi|^2+m^2\right)^{\sigma}|\hat{h}(\xi)|^2\dd\xi\nonumber\\
&\geq \left\{\begin{array}{l}\label{partial} \mathcal{K}(\Phi_\sigma)m^{2\sigma}\displaystyle\int_{\mathbb{R}^{N}}|u(0,y)|^2\dd y,\\
\mathcal{K}(\Phi_\sigma)\displaystyle\int_{\mathbb{R}^{N}}\left(4\pi^2|\xi|^2\right)^{\sigma}|\hat{h}(\xi)|^2\dd\xi=\mathcal{K}(\Phi_\sigma)\displaystyle\int_{\mathbb{R}^{N}}\left|(-\Delta)^\sigma u(0,y)\right|^2\dd y\end{array}\right.\\
&\geq C\left(\displaystyle\int_{\mathbb{R}^N}|u(0,y)|^2\dd y+\int_{\mathbb{R}^{N}}\int_{\mathbb{R}^{N}}\frac{|u(0,y)-u(0,w)|^2}{|y-w|^{N+2\sigma}} \dd w\dd y\right),\label{immersion0}
\end{align}
\end{subequations}
see \cite{Guide} for the definition of the fractionary $\sigma$-Laplacian operator.

We now translate our results in terms of function spaces and their norms. For this, let us start considering the Hilbert space
\[H^1(\mathbb{R}^{N+1}_+,x^{1-2\sigma})=\left\{u\colon \mathbb{R}^{N+1}_+\to\mathbb{R}: \iint_{\mathbb{R}^{N+1}_+}\left(|\nabla u|^2+|u|^2\right)x^{1-2\sigma}\dd y\dd x<\infty \right\}\]
endowed with the norm
\[\|u\|_\sigma=\left(\iint_{\mathbb{R}^{N+1}_+}\left(|\nabla u(x,y)|^2+|u(x,y)|^2\right)x^{1-2\sigma}\dd y\dd x\right)^{\frac{1}{2}}.\]
Observe that we have the immersion
\begin{equation}\label{immersionloc}H^1(\mathbb{R}^{N+1}_+,x^{1-2\sigma})\hookrightarrow L^q_{loc}(\mathbb{R}^{N+1}_+)\quad\text{for any}\quad q\in [2,2^*],\end{equation}
where $2^*=2(N+1)/(N-1)$.

Let $v\colon\mathbb{R}^N\to \mathbb{R}$ be a measurable function. For any $\sigma\in (0,1)$ we denote
\[[v]^2_{\sigma,2}=\int_{\mathbb{R}^N}\int_{\mathbb{R}^N}\frac{|v(x)-v(y)|^2}{|x-y|^{N+2\sigma}}\dd x\dd y.\]
We also denote by $|\cdot|_q$ and $|\cdot|_{\sigma,q}$
the usual norm in $L^q(\mathbb{R}^N)$ and $L^q(\mathbb{R}^{N+1}_+,x^{1-2\sigma})$, respectively.
Now we consider the space
\[W^{\sigma,2}(\mathbb{R}^N)=\left\{v\in L^2(\mathbb{R}^N)\,:\,[v]^2_{\sigma,2}<\infty\right\}\]
endowed with the norm
\[\|v\|_{N,\sigma}=\left(|v|^2_2+[v]^2_{\sigma,2}\right)^{1/2},\]
which makes $W^{\sigma,2}(\mathbb{R}^N)$ a reflexive Banach space, see  \cite{Demengel}. The term $[v]_{\sigma,2}$ is the Gagliardo seminorm of $v$. It is well-known that $W^{\sigma,2}(\mathbb{R}^N)=H^\sigma(\mathbb{R}^N)$, see \cite{Demengel,Guide}.

Translating \eqref{immersion0} in terms of the norms of $H^1(\mathbb{R}^{N+1}_+,x^{1-2\sigma})$ and $W^{\sigma,2}(\mathbb{R}^N)$, it means that
\begin{align}\label{immersion1}
\left(\int_{\mathbb{R}^n}|u(0,y)|^2\dd y\right)^{1/2}&\leq C'\left(\iint_{\mathbb{R}^{N+1}_+}\left(|\nabla u(x,y)|^2+|u|^2\right)x^{1-2\sigma}\dd y\dd x\right)^{1/2}\nonumber\\
&=C'\|u\|_\sigma,
\end{align}
for all $u\in H^1(\mathbb{R}^{N+1}_+,x^{1-2\sigma})$. Inequality \eqref{immersion1} goes back to J.L. Lions \cite{Lions}, see \cite[proof of Proposition 2.1]{Capella}.

According to Brändle et al \cite[Theorem 2.1]{Brandle} or Xiao \cite[Theorem 1.1]{Xiao}, we also have
\[\left(\int_{\mathbb{R}^n}|u(0,y)|^{2N/(N-2\sigma)}\right)^{(N-2\sigma)/N}\leq C\iint_{\mathbb{R}^{N+1}_+}|\nabla u(x,y)|^2x^{1-2\sigma}\dd x\dd y.\]
(The value of the best constant $C$ is explicitly given in those papers.) Therefore,
\begin{equation}\label{immersion1a}\left(\int_{\mathbb{R}^n}|u(0,y)|^{2N/(N-2\sigma)}\right)^{(N-2\sigma)/2N}\leq C\|u\|_\sigma.\end{equation}
Interpolating, we obtain for all $q\in [2,2^*_\sigma]$
\begin{align}\label{immersion1aa}
|u(0,\cdot)|_q=\left(\int_{\mathbb{R}^N}|u(0,y)|^q\dd y\right)^{1/q}\leq C''\|u\|_\sigma,\end{align}
where the constant $C''$ depends on $q$ and
\[2^{*}_\sigma= \frac{2N}{N-2\sigma}.\]

Inequality \eqref{partial}, in the case $m=1$, translates into
\begin{equation}\label{immersion1b}
\int_{\mathbb{R}^N}|u(0,y)|^2\dd y\leq \frac{1}{\mathcal{K}(\Phi_{\sigma})}\|u\|^2_\sigma.
\end{equation}

Since we suppose $\sigma\in (0,1)$ and we have $2\sigma<N$, the immersion
\begin{equation}\label{immersion2}
W^{\sigma,2}(\mathbb{R}^N)\hookrightarrow L^q(\mathbb{R}^N)\end{equation}
is continuous for any $q\in [2,2^*_\sigma]$, see \cite[Theorem 4.47]{Demengel}.
The space $W^{\sigma,2}(\Omega)$ is well-defined for an open, bounded set $\Omega\subset \mathbb{R}^N$. In the sequel, we suppose $\Omega$ to have Lipschitz boundary, so that 
denoting
\[[v]^2_{W^{\sigma,2}(\Omega)}=\int_\Omega\int_\Omega\frac{|v(x)-v(y)|^2}{|x-y|^{N+2\sigma}}\dd x\dd y\]
and
\begin{align*}W^{\sigma,2}(\Omega)&=\left\{v\in L^2(\Omega)\,:\,[v]^2_{W^{\sigma,2}(\Omega)}<\infty\right\},
\end{align*}
then $W^{\sigma,2}(\Omega)$ is a reflexive Banach space (see, e.g., \cite{Demengel} and \cite{Guide}) endowed with the norm
\[\|u\|_{W^{\sigma,2}(\Omega)}=|u|_2+[u]_{W^{\sigma,2}(\Omega)}.\]

We also have that the immersion
\begin{equation}\label{immersion3}
W^{\sigma,2}(\Omega)\hookrightarrow L^q(\Omega)\end{equation} is compact for any $q\in \left[1,2^{*}_\sigma\right)$, see \cite[Theorem 4.54]{Demengel}. As usual, the immersion $W^{\sigma,2}(\Omega)\hookrightarrow L^{2^{*}_\sigma}(\Omega)$ is continuous: see \cite[Corollary 4.53]{Demengel}. We denote the norm in the space $L^q(\Omega)$ by $|\cdot|_{L^q(\Omega)}$.

\section{Preliminaries}\label{Prelim}
We denote by $u(0,y)$ the trace of $u\in H^1(\mathbb{R}^{N+1}_+)$ in $\mathbb{R}^N$.
\begin{definition}\label{wsolution} We say that $u \in H^1(\mathbb{R}^{N+1}_+)$ is a weak solution of \eqref{P} if, and only if, we have
\begin{multline*}
\iint_{\mathbb{R}^{N+1}_+}\left(\nabla u\nabla \varphi + m^2u\varphi\right)x^{1-2\sigma}\dd y\dd x\ +\ \int_{\mathbb{R}^N} V(y)u(0,y)\varphi(0,y)\dd y\phantom{\hspace{2cm}}
\end{multline*}
\begin{align}\label{derivative} = \int_{\mathbb{R}^N}\big[W*F(u(0,y))\big]f(u(0,y))\varphi(0,y) \dd y,
\end{align}
for all $\varphi \in H^1(\mathbb{R}^{N+1}_+, x^{2\sigma-1})$.
\end{definition}
Since the integration variables are clear from the context, we usually omit $\dd x$ and $\dd y$.

The functional $I\colon H^1(\mathbb{R}^{N+1}_+)\to \mathbb{R}$ defined by
\begin{align}\label{I} I(u) &= \frac{1}{2}  \iint_{\mathbb{R}^{N+1}_+} \left(|\nabla u|^2 + m^2u^2\right)x^{1-2\sigma} +\frac{1}{2}\int_{\mathbb{R}^N} V(y)|u(0,y)|^2\nonumber\\ &\qquad-\frac{1}{2}\int_{\mathbb{R}^N}\big[W*F(u(0,y))\big]F(u(0,y)), 
\end{align}
describes the ``energy'' naturally attached to problem \eqref{P}. As usual, $I$ is a $C^1$ functional and critical points of $I$ are weak solutions \eqref{P}.

\begin{remark}\label{obs1}
It follows from \textup{($f1$)} and \textup{($f2$)} that, for any fixed $\xi>0$, there exists a constant $C_\xi$ such that
\begin{equation}\label{boundf}|f(t)|\leq\xi t+C_\xi t^{\theta-1},\quad\forall\ t\geq 0\end{equation}
and
\begin{equation}\label{boundF}|F(t)|\leq\xi t^2+C_\xi t^{\theta}\leq C(t^2+t^\theta),\quad\forall\ t\geq 0.\end{equation}
Observe that $u(0,y)\in L^\theta(\mathbb{R}^{N})$ and $u(0,y)\in L^2(\mathbb{R}^{N})$ imply $F(u(0,y))\in L^1(\mathbb{R}^{N})$.
\end{remark}

\begin{proposition}[Hausdorff-Young]\label{HYoung}
Assume that, for $1\leq p, r, s \leq \infty$, we have $f\in L^p (\mathbb{R}^{N})$, $g\in  L^r (\mathbb{R}^{N})$ and
\[\frac{1}{p}+\frac{1}{r}= 1 +\frac{1}{s}.\] Then
\[|f*g|_{s} \leq |f|_{p}|g|_{r}.\]
\end{proposition}

We now handle the existence of the energy functional \eqref{I}. We denote by $L^q_w(\mathbb{R}^{N})$ the weak $L^q$ space and by $|\cdot|_{q_w}$ its usual norm (see \cite{Lieb}). The next result is a generalized version of the Hardy-Littlewood-Sobolev inequality:
\begin{proposition}[Lieb \cite{Lieb}]\label{pLieb} Assume that $p,q,r\in(1,\infty)$ and \[\frac{1}{p}+\frac{1}{q}+\frac{1}{r}=2.\]
	Then, for some constant $N=N(p,q,r)>0$ and for any $f\in L^p(\mathbb{R}^{N})$, $g\in L^r(\mathbb{R}^{N})$ and
	$h\in L^q_w(\mathbb{R}^{N})$, we have the inequality
	\[\int_{\mathbb{R}^{N}}\int_{\mathbb{R}^{N}}f(t)h(t-s)g(s)\dd t\dd s\leq N|f|_{p}|g|_{r} |h|_{q_w}.\]
\end{proposition}

\begin{lemma}\label{estconv} For a positive constant $C$ holds
\[\left|\frac{1}{2}\int_{\mathbb{R}^{N}}\big[W*F(u(0,y))\big]F(u(0,y))\right|\leq  C\left(\|u\|^{2}_{\sigma}+\|u\|^{\theta}_{\sigma}\right)^2.\]
\end{lemma}

\noindent\begin{proof}Let us denote
\[\Psi(u)=\frac{1}{2}\int_{\mathbb{R}^{N}}\big[W*F(u(0,y))\big]F(u(0,y)).\]
Since $W=W_1+W_2$,
\begin{align*}\Psi(u)&=\frac{1}{2}\int_{\mathbb{R}^{N}}\big[W_1*F(u(0,y))\big]F(u(0,y))+\frac{1}{2}\int_{\mathbb{R}^{N}}\big[W_2*F(u(0,y))\big]F(u(0,y))\nonumber\\
&=:J_1+J_2\end{align*}

Let us suppose that $|u(0,y)|^\theta\in L^t(\mathbb{R}^{N})$ for some $t\geq 1$. Then $|u(0,y)|^2\in L^t(\mathbb{R}^{N})$ and $F(u(0,y))\in L^t(\mathbb{R}^{N})$ (as consequence of \eqref{boundF}). Application of Proposition \ref{pLieb} yields
\[|J_1|=\left|\frac{1}{2}\int_{\mathbb{R}^{N}}\left[W_1*F(u(0,y))\right]\,F(u(0,y))\right|\leq N\,|W_1|_{r}|F(u(0,y))|_{t}|F(u(0,y))|_t.\]
Since
$\frac{1}{r}+\frac{2}{t}=2$ implies $t=\frac{2r}{2r-1}$, we have
\begin{align}\label{J1} |J_1| &\leq C|F(u(0,y))|_{\frac{2r}{2r-1}}|F(u(0,y)))|_{\frac{2r}{2r-1}} \leq C'\left(|u(0,y)|^2_2+|u(0,y)|^\theta_\theta\right)^2\nonumber\\
&\leq C'' ( \left\| u \right\|^2_{\sigma}+\left\| u \right\|^{\theta}_{\sigma})^2 <\infty,
\end{align}
by \eqref{immersion1aa}. (Observe that, in order to apply the immersion \eqref{immersion1aa}, we must have $t\theta<2N/(N-2\sigma)$, that is, $r>N/[N(2-\theta)+2\sigma \theta]$.

In the case $W_2\in L^\infty(\mathbb{R}^{N})$ we can take $t=1$, therefore
\begin{align}\label{J2}
|J_2|&=\left|\frac{1}{2}\int_{\mathbb{R}^{N}}\big[W_2*F(u(0,y))\big]F(u(0,y))\right|\leq C\left(|u(0,y)|^2_{2}+|u(0,y)|^\theta_{\theta}\right)^2\nonumber\\
&\leq
C''\left(\|u\|^2_{\sigma}+\|u\|^\theta_{\sigma}\right)^2.
\end{align}

From \eqref{J2} and \eqref{J1} results the claim.
$\hfill\Box$\end{proof}\vspace*{.3cm}

From Lemma \ref{estconv} follows immediately that the functional $I(u)$ is well-defined.

\section{Ground state}\label{GS}
Because we are looking for a \emph{positive} ground state solution, we suppose that $f(t)=0$ for $t<0$. Observe that our version of the mountain pass geometry is valid for all $u\in H^1(\mathbb{R}^{N+1}_+,x^{1-2\sigma})$ such that $u_+=\max\{u,0\}\neq 0$.
\begin{lemma}\label{gpm}
$I$ satisfies the mountain pass theorem geometry. More precisely,
\begin{enumerate}
\item [$(i)$] There exist $\rho,\delta>0$ such that $I|_S\geq \delta>0$ for all $v\in S$, where
\[S=\left\{v\in H^1(\mathbb{R}^{N+1}_+,x^{1-2\sigma})\,:\, \|v\|_\sigma=\rho\right\}.\]
\item [$(ii)$] For any fixed $u\in H^1(\mathbb{R}^{N+1}_+,x^{1-2\sigma})$ such that $u_+\neq 0$, there exists  $\tau\in \mathbb{R}$ such that $\|\tau u\|>\rho$ and  $I(\tau u) <0$.
\end{enumerate}
\end{lemma}

\noindent \begin{proof}Decomposing $I(u)$ into three terms
	\begin{align}
		I(u) =&\frac{1}{2}\iint_{\mathbb{R}^{N+1}_+}\left(|\nabla u|^2+m^2u^2\right)x^{1-2\sigma}+\frac{1}{2}\int_{\mathbb{R}^{N}}V(y)[u(0,y)]^2\nonumber\\
		&\quad-\frac{1}{2}\int_{\mathbb{R}^{N}}\left[W*F(u(0,y))\right]F(u(0,y))\nonumber\\
		=&:\frac{1}{2}I_1+\frac{1}{2}I_2-\Psi(u)\label{energy}
		\end{align}
with $\psi(u)$ defined in Lemma \ref{estconv}, let us consider the individual terms.

Observe that hypotheses ($V1$) implies
\begin{align}\label{I2<}
I_2&=\int_{\mathbb{R}^{N}}V(y)|u(0,y)|^2\leq C_1\int_{\mathbb{R}^{N}}|u(0,y)|^2
\end{align}
and also
\begin{align}\label{I2>}
I_2&=\int_{\mathbb{R}^{N}}(V(y)+V_0)|u(0,y)|^2-V_0\int_{\mathbb{R}^{N}}|u(0,y)|^2\nonumber\\
&\geq -V_0\int_{\mathbb{R}^{N}}|u(0,y)|^2\geq -\frac{V_0}{\mathcal{K}(\Phi_{\sigma})}\|u\|^2_\sigma,
\end{align}
as consequence of the immersion \eqref{immersion1b}.

Substituting the last inequality into \eqref{energy} we obtain
\begin{align*}
\frac{1}{2}\left(I_1(u)+I_2(u)\right)&\geq \frac{1}{2}\iint_{\mathbb{R}^{N+1}_+}\left(|\nabla u|^2+m^2u^2\right)x^{1-2\sigma}-\frac{V_0}{2\mathcal{K}(\Phi_{\sigma})}\|u\|^2\\
&\geq\frac{1}{2}\mathcal{C}\|u\|^2_\sigma-\frac{V_0}{2\mathcal{K}(\Phi_\sigma)}\|u\|^2_\sigma\,,
\end{align*}
where $\mathcal{C}=\min\{1,m^2\}>0$. It follows from hypotheses ($V1$) that there exists a constant $\mathcal{C}'$ such that
\begin{align}\label{I+}\frac{1}{2}\left(I_1(u)+I_2(u)\right)\geq \mathcal{C}'\|u\|^2_\sigma\end{align}
and so $I(u)\geq \mathcal{C}'\|u\|^2-\psi(u)\geq \mathcal{C}'\|u\|^2_\sigma-C\left(\|u\|^2_\sigma+\|u\|^\theta_\sigma\right)^2$, thus implying ($i$) when we choose $\rho>0$ small enough.

In order to prove ($ii$), fix $u\in H^1(\mathbb{R}^{N+1}_+,x^{1-2\sigma})\setminus\{0\}$ such that $u_+\neq 0$. For all $t>0$ consider the function $g_{u}\colon(0,\infty)\to\mathbb{R}$ defined by
\[g_{u}(t)=\Psi\left(\frac{tu(0,y)}{\|u\|}\right)\]
where $\Psi$ was defined before.

An easy calculation shows that
\begin{align*}
g'_{u}(t)&=\frac{2}{t}\int_{\mathbb{R}^{N}}
\left[W*F\left(\frac{tu(0,y)}{\|u\|}\right)\right]\frac{f}{2}\left(\frac{tu(0,y)}{\|u\|}\right)\left(\frac{tu(0,y)}{\|u\|}\right)
\geq\frac{4}{t}g_{u}(t),\end{align*}
the last inequality being a consequence of the Ambrosetti-Rabinowitz inequality. Observe that $g'_{u}(t)>0$ for $t>0$.

Thus, we obtain
\begin{align}\label{H}
\Psi(\tau u)=g_{u}(\tau\|u\|_\sigma)\geq D\left(\tau\|u\|_\sigma\right)^{4}.\end{align}
for a constant $D>0$.

Gathering \eqref{I2<}, \eqref{immersion1b} and \eqref{H}, we have
\begin{align*}I(\tau u)&\leq C\tau^2\|u\|^2_\sigma-D\tau^4\|u\|^4_\sigma
\end{align*}
for a constant $C$. Thus, it suffices to take $e=\tau u$ for any $u_+\neq 0$ and $\tau$ large enough.
$\hfill\Box$\end{proof}\vspace*{.4cm}

Results from the mountain pass theorem without the PS condition the existence of a Palais-Smale sequence $(u_n)\subset H^1(\mathbb{R}^{N+1}_+,x^{1-2\sigma})$ such that
\[I'(u_n)\to 0\qquad\textrm{and}\qquad I(u_n)\to c,\]
where
\[c=\inf_{\alpha\in \Gamma}\max_{t\in [0,1]}I(\alpha(t)),\]
and $\Gamma=\left\{\alpha\in C^1\left([0,1],H^1(\mathbb{R}^{N+1}_+,x^{1-2\sigma})\right)\,:\,\alpha(0)=0,\,\alpha(1)=e<0\right\}$. \vspace*{.2cm}

We now consider the Nehari manifold
\begin{align*}\mathcal{N}&=\left\{u\in H^1(\mathbb{R}^{N+1}_+,x^{1-2\sigma})\setminus\{0\}\,:\,\langle I'(u), u\rangle=0\right\}.\\
\end{align*}

The next result follows immediately from our estimates:
\begin{lemma}\label{lN}
There exists $\beta>0$ such that $\|u\|_\sigma\geq \beta$ for all $u\in \mathcal{N}$.
\end{lemma}

An alternative characterization of $c$ is given in terms of the Nehari manifold: there exists a unique $t_u=t(u)>0$ such that $\Phi'_u(t)>0$ for $t<t_u$ and $\Phi'_u(t)<0$ for $t>t_u$. Furthermore, $\Phi'_u(t_u)=0$ implies that $t_uu\in \mathcal{N}$. The map $u\mapsto t_u$ ($u_+\neq 0$) is continuous and $c=c^*$, where
\[c^*=\inf_{u\in H^1(\mathbb{R}^{N+1}_+,x^{1-2\sigma})\setminus\{0\}}\max_{t\geq 0} I(tu).\] For details, see \cite[Section 3]{Rabinowitz} or \cite{Felmer}.

Standard arguments prove the next affirmative:
\begin{lemma}\label{boundedseq}
Let $(u_n)\subset H^1(\mathbb{R}^{N+1}_+,x^{1-2\sigma})$ be a sequence such that $I(u_n)\to c^*$ and $I'(u_n)\to 0$, where
\[c^*=\inf_{u\in H^1(\mathbb{R}^{N+1}_+,x^{1-2\sigma})\setminus\{0\}}\max_{t\geq 0}I(tv).\]
Then $(u_n)$ is bounded and $($for a subsequence$)$ $u_n\rightharpoonup u$ in $H^1(\mathbb{R}^{N+1}_+,x^{1-2\sigma})$.
\end{lemma}

\begin{lemma}\label{lK}
Let $U\subseteqq \mathbb{R}^{N}$ be any open set. For $1<p<\infty$, let $(f_n)$ be a bounded sequence in $L^p(U)$ such that $f_n(z)\to f(z)$ a.e. in $\mathbb{R}^{N}$. Then $f_n\rightharpoonup f$.
\end{lemma}
The proof of Lemma \ref{lK} can be found, e.g., in \cite[Lemme 4.8, Chapitre 1]{Kavian}.\vspace*{.2cm}

We denote by $c_\infty$ the mountain pass level attached to problem \eqref{P} with $V(y)$ changed by the constant potential $V_\infty>0$. Precisely,
\[c_\infty=\inf_{u\in H^1(\mathbb{R}^{N+1}_+,x^{1-2\sigma})\setminus\{0\}}\max_{t\geq 0} I_\infty(tu),\]
where
\begin{align*}I_\infty(u) =&\frac{1}{2}\iint_{\mathbb{R}^{N+1}_+}\left(|\nabla u|^2+m^2u^2\right)x^{1-2\sigma}+\frac{1}{2}\int_{\mathbb{R}^{N}}V_\infty[u(0,y)]^2\nonumber\\
&\quad-\frac{1}{2}\int_{\mathbb{R}^{N}}\left[W*F(u(0,y))\right]F(u(0,y)).
\end{align*}
(This characterization of the mountain pass level attached to the potential $V_\infty$ is obtained as before, introducing the Nehari manifold.)
\begin{proposition}\label{asymptotic}
Assuming that hypotheses \textup{(f1), (f2), (f3)} and \textup{$(W_h)$} are valid, for $(x,y)\in (0,\infty)\times \mathbb{R}^N= \mathbb{R}^{N+1}_+$ the problem
		\begin{equation}
		\left\{\begin{array}{ll}
		\Delta_y u+\frac{1-2\sigma}{x}u_x+u_{xx}-m^2u=0 &\text{in }\ \mathbb{R}^{N+1}_+ \\ \\
		\displaystyle\lim_{x\to 0^+}\left(-x^{1-2\sigma}\frac{\partial u}{\partial x}\right)= -V_\infty u+ \left[W*F(u)\right]f(u), &\text{in}\  \left\{ 0\right\} \times \mathbb{R}^N \simeq\mathbb{R}^N,
		\end{array}\right.\tag{$P_\infty$}\label{Pinfty}
		\end{equation}	
has a positive ground state solution for any constant potential $V_\infty>0$.
\end{proposition}
\begin{proof} Let $(u_n)$ be the minimizing sequence given by Lemma \ref{gpm} in the case of $V_\infty$ instead of $V(y)$. That is,
\[I'_\infty(u_n)\to 0\qquad\textrm{and}\qquad I_\infty(u_n)\to c_\infty.\]
Then, there exist $R,\delta>0$ and a sequence $(z_n)\subset\mathbb{R}^{N}$ such that \begin{equation}\label{Lions}\liminf_{n\to\infty}\int_{B_R(z_n)}|u_n(0,y)|^2\geq \delta.\end{equation}
If false, a result obtained by Lions (see \cite{CC}) guarantees that $u_n(0,\cdot)\to 0$ in $L^q(\mathbb{R}^{N})$ for $2<q<2^*_\sigma$, thus implying that
\[\int_{\mathbb{R}^N} \left[W*F(u_n(0,y))\right]f(u_n(0,y))u_n(0,y)  \to 0,\] contradicting Lemma \ref{lN}.

We define
\[w_n(x,y)=u_n(x,y-z_n)
.\]
From \eqref{Lions} we derive that
\[\int_{B_R(0)}|w_n(0,y)|^2\geq \frac{\delta}{2}.\]

Since $I_\infty$ and $I'_\infty$ are both invariant by translation, it also holds that
\[I'_\infty(w_n) \to 0\quad\textrm{ and }\quad I_\infty(w_n) \to c^*.\]

It follows, as before, that $(w_n)$ is bounded and therefore, $w_n\rightharpoonup w$ for a subsequence. A standard reasoning proves that $w_n(z)\to w(z)$ a.e. in  $(\mathbb{R}^{N+1}_+)$, $w_n\to w$ in $L^s_{loc}(\mathbb{R}^{N+1}_+)$ for all $s\in [2,2^*)$, $w_n(0,y)\to w(0,y)$ a.e. in $(\mathbb{R}^{N})$
and $w_n(0,\cdot)\to w(0,\cdot)$ in $L^q_{loc}(\mathbb{R}^{N})$, for all $q\in [2,2^{*}_\sigma)$.

For $\varphi \in C^\infty_0(\mathbb{R}^{N+1}_+)$ arbitrary, let us consider $\psi_n=(w_n - w)\varphi\in H^1(\mathbb{R}^{N+1}_+)$. We have
\begin{align}\label{testfunction}
\langle I'_\infty(w_n) , \psi_n\rangle&= \iint_{\mathbb{R}^{N+1}_+} \nabla w_n \cdot \nabla\psi_n\,x^{1-2\sigma}+ \iint_{\mathbb{R}^{N+1}_+}m^2w_n\psi_n\,x^{1-2\sigma}\nonumber\\
{}&\qquad+ \int_{\mathbb{R}^N} V_\infty w_n(0,y)\psi_n(0,y) \nonumber \\
{}&\qquad-\int_{\mathbb{R}^N}\left[W * F(w_n(0,y))\right] f(w_n(0,y))) \psi_n(0,y)\nonumber\\
&=J_1+J_2+J_3-J_4.
\end{align}
We start considering
\begin{align*}
J_4=\int_{\mathbb{R}^N}\left[W * F(w_n(0,y))\right]f(w_n(0,y))\psi_n(0,y).
\end{align*}
Because $\displaystyle\lim_{n\to\infty}\langle I'_\infty(w_n) , (w_n - w)\varphi \rangle=0$, it follows from \cite[Lemma 3.5]{Ackermann} that $J_4\to 0$ when $n\to\infty$ and thus is easily verified that $J_2+J_3-J_4\to 0$ when $n\to\infty$.

We now consider $J_1$:
\begin{align*}
J_1&=\iint_{\mathbb{R}^{N+1}_+}\nabla w_n \cdot \nabla ( (w_n-w)\varphi)\,x^{1-2\sigma}\\
&=\iint_{\mathbb{R}^{N+1}_+} \nabla w_n \cdot  \varphi\nabla (w_n - w)\,x^{1-2\sigma} +
\iint_{\mathbb{R}^{N+1}_+} \nabla w_n\cdot(w_n-w) \nabla \varphi\,x^{1-2\sigma}\\
&=\iint_{\mathbb{R}^{N+1}_+}\left(|\nabla(w_n-w)|^2\varphi+\varphi \nabla v\cdot \nabla(w_n-w)+\nabla w_n\cdot(w_n-w) \nabla \varphi\right)x^{1-2\sigma}.
\end{align*}
We infer that
\begin{align*}\lim_{n\to\infty}\iint_{\mathbb{R}^{N+1}_+} |\nabla(w_n-w)|^2 \varphi\,x^{1-2\sigma} &=-\lim_{n\to\infty} \iint_{\mathbb{R}^{N+1}_+}
\varphi\nabla w \cdot\nabla(w_n-w)\,x^{1-2\sigma} \\
&\qquad -  \lim_{n\to\infty}\iint_{\mathbb{R}^{N+1}_+} (w_n-w)\nabla w_n \cdot\nabla
\varphi\,x^{1-2\sigma}.\end{align*}
Since
\begin{align*}
\lim_{n\to\infty} \iint_{\mathbb{R}^{N+1}_+} \varphi\nabla w\cdot \nabla( w_n - w)\,x^{1-2\sigma} &=0\\
\intertext{and}
\lim_{n\to\infty}\iint_{\mathbb{R}^{N+1}_+} (w_n - w)\nabla w_n \cdot \nabla \varphi\,x^{1-2\sigma} &=0
\end{align*}
(because $\nabla w_n$ is bounded), we deduce that
\[ \nabla w_n \rightarrow \nabla w \quad \ \mbox{a.e. in} \quad \mathbb{R}^{N+1}_+.\]
Thus
\[\langle I'_\infty(w),w\rangle =0\]
and $w \in \mathcal{N}$.

We now turn our attention to the positivity of $w$. Since
\begin{multline*}\iint_{\mathbb{R}^{N+1}_+}\left(\nabla w\cdot \nabla\varphi+m^2w\varphi\right)x^{1-2\sigma}+\int_{\mathbb{R}^{N}}V_\infty w(0,y)\varphi(0,y)\\
=\int_{\mathbb{R}^{N}}\left[W*F(w(0,y))\right]f(w(0,y))\varphi(0,y),\end{multline*}
choosing $\varphi=w_-$, the right-hand side of the equality is positive by equations \eqref{energy} and \eqref{I+}, since $J_1+J_2+J_3\geq K\|w\|^2_\sigma$),  while $\Psi(w)=J_4\leq 0$. 

By Theorem \ref{t1} we have $w(x,y)\geq 0$. Taking a compact $K\subset H^1(\mathbb{R}^{N+1}_+,x^{1-2\sigma})$, we apply \cite[Theorem 8.20, Corollary 8.21]{GT} to conclude that $w$ is strictly positive. An alternative proof follows by noting that the weight $x^{1-2\sigma}$ satisfies condition (1.2) in Gutierrez \cite{Gutierrez} and therefore the positivity of $w$ follows from Theorems 3.1 and 4.5 in that paper. We are done.
$\hfill\Box$\end{proof}\vspace*{.3cm}

In order to consider the general case of the potential $V(y)$, we state a well-known result due to M. Struwe, with the notation adapted to our case:
\begin{lemma}[Splitting Lemma]\label{Struwe} Let $(u_n)\subset H^1(\mathbb{R}^{N+1}_+,x^{1-2\sigma})$ be such that
	\[I(u_n)\to c,\qquad I'(u_n)\to 0\]
	and $u_n\rightharpoonup u$ weakly on $X$. Then $I'(u)=0$ and we have \emph{either}
	\begin{enumerate}
		\item [($i$)] $u_n\to u$ strongly on $X$;
		\item [($ii$)] there exist $k\in\mathbb{N}$, $(y^j_n)\in\mathbb{R}^N$ such that $|y^j_n|\to\infty$ for $j\in \{1,\ldots,k\}$ and nontrivial solutions $u^1,\ldots,u^k$ of problem \eqref{Pinfty} so that
		\[I(u_n)\to I(u_0)+\sum_{j=1}^k I_\infty (u_j)\]
		and 
		\[\left\|u_n-u_0-\sum_{j=1}^ku^j(\cdot-y^j_n)\right\|\to 0.\]
	\end{enumerate}
	
\end{lemma}

\begin{lemma}\label{PS}
	The functional $I$ satisfies $(PS)_c$ for any $0\leq c<c_\infty$.
\end{lemma}
\begin{proof}Let us suppose that $(u_n)$ satisfies
	\[I(u_n)\to c<c_\infty\qquad\text{and}\qquad I'(u_n)\to 0.\]
	By adapting Lemma \ref{boundedseq} to the functional $I_\infty$, we can suppose that the sequence $(u_n)$ is bounded,. Therefore, for a subsequence, we have $u_n\rightharpoonup u$ in $H^1(\mathbb{R}^{N+1}_+)$. It follows from the Splitting Lemma (Lemma \ref{Struwe}) that $I'(u)=0$. Since 
	\begin{align*}
	I'(u)\cdot u&=\iint_{\mathbb{R}^{N+1}_+}\left(|\nabla u|^2+m^2u^2\right)x^{1-2\sigma}+\int_{\mathbb{R}^{N}}V(y)|u(0,y)|^2\\
	&\qquad-\int_{\mathbb{R}^{N}}[W*F(u(0,y))]f(u(0,y))u(0,y)\\
	\intertext{and}
	I(u)&=\frac{1}{2}\iint_{\mathbb{R}^{N+1}_+}\left(|\nabla u_0|^2+m^2u^2_0\right)x^{1-2\sigma}+\frac{1}{2}\int_{\mathbb{R}^{N}}V(y)|\gamma(u_0)|^2\\
	&\qquad-\frac{1}{2}\int_{\mathbb{R}^{N}}[W*F(\gamma(u_0))]F(\gamma(u_0)),
	\end{align*}
	we conclude that 
	\begin{equation}\label{Iu0}
	I(u)=\int_{\mathbb{R}^{N}}[W*F(u(0,y))]\left(\frac{1}{2}f(u(0,y))u(0,y)-F(u(0,y))\right)>0,\end{equation}
	as consequence of the Ambrosetti-Rabinowitz condition.
	
	If $u_n\not\to u$ in $H^1(\mathbb{R}^{N+1}_+,x^{1-2\sigma})$, by applying again the Splitting Lemma we guarantee the existence of $k\in\mathbb{N}$ and nontrivial solutions $u^1,\ldots,u^k$ of problem \eqref{Pinfty} satisfying 
	\[\lim_{n\to\infty}I(u_n)=c=I(u_0)+\sum_{j=1}^kI_\infty(u^j)\geq kc_\infty\geq c_\infty\]
	contradicting our hypothesis. We are done. 
	$\hfill\Box$\end{proof}\vspace*{.2cm}

We prove the next result by adapting the proof given in Furtado, Maia e Medeiros \cite{FMM}:
\begin{lemma}\label{ccinfty} We have	\[0<c^*<c_\infty,\]
	where $c^*$ is characterized in Lemma \ref{boundedseq}.
\end{lemma}
\begin{proof}Let $\bar u\in \mathcal{N}_\infty$ be the weak solution of \eqref{Pinfty} given by Proposition \ref{asymptotic} and $t_{\bar u}>0$ be the unique number such that $t_{\bar u}\bar u\in \mathcal{N}$. We claim that $t_{\bar u}<1$. Indeed,
	\[\int_{\mathbb{R}^{N}}[W*F(\gamma(t_{\bar u}\bar u))]f(\gamma(t_{\bar u}\bar u))\gamma(t_{\bar u}\bar u)\hspace*{7cm}\]
	
	\vspace*{-.5cm}\begin{align*}
	&=t^2_{\bar u}\iint_{\mathbb{R}^{N+1}_+}\left(|\nabla {\bar u}|^2+m^2{\bar u}^2\right)+\int_{\mathbb{R}^{N}}V(y)|\gamma(\bar{u})|^2\\
	&< t^2_{\bar u}\iint_{\mathbb{R}^{N+1}_+}\left(|\nabla {\bar u}|^2+m^2{\bar u}^2\right)+\int_{\mathbb{R}^{N}}V_\infty|\gamma(\bar{u})|^2\\
	&=t^2_{\bar u}\int_{\mathbb{R}^{N}}[W*F(\gamma(\bar u))]f(\gamma(\bar u))\gamma(\bar u)\\
	&=t^2_{\bar u}\left(\int_{\mathbb{R}^{N}}[W*F(\gamma(\bar u))]f(\gamma(\bar u))\gamma(\bar u)+\int_{\mathbb{R}^{N}}[W*F(\gamma(t_{\bar u}\bar u))]f(\gamma(\bar u))\gamma(\bar u)\right.\\
	&\qquad\quad\left.-\int_{\mathbb{R}^{N}}[W*F(\gamma(t_{\bar u}\bar u))]f(\gamma(\bar u))\gamma(\bar u)\right)
	\end{align*}
	thus yielding
	\begin{align*}
	0&>\int_{\mathbb{R}^{N}}[W*F(\gamma(t_{\bar u}\bar u))]\left(\frac{f(\gamma(t_{\bar u}\bar u))}{\gamma(t_{\bar u}\bar u)}-\frac{f(\gamma(\bar u))}{\gamma(\bar u)}\right)\\
	&\qquad +t^2_{\bar u}\int_{\mathbb{R}^{N}}\left[W*\left(F(\gamma(t_{\bar u}\bar u))-F(\gamma(\bar u))\right)\right]f(\gamma(u))\gamma(u).
	\end{align*}
	
	If $t_{\bar u}\geq 1$, since $f(s)/s$ is increasing, the first integral is non-negative and, since $F$ is increasing, the second integral as well. We conclude that $t_{\bar u}<1$.
	
	Lemma \ref{boundedseq} and its previous comments show that 
	\[c\leq \max_{t\geq 0}I(t\bar u)=I(t_{\bar u}\bar u)=\int_{\mathbb{R}^{N}}[W*F(\gamma(t_{\bar u}\bar u))]\left(\frac{1}{2}f(\gamma(t_{\bar u}\bar u))\gamma(t_{\bar u}\bar u)-F(\gamma(t_{\bar u}\bar u))\right).\]
	Since 
	\[g(t)=\int_{\mathbb{R}^{N}}[W*F(\gamma(t\bar u))]\left(\frac{1}{2}f(\gamma(t\bar u))\gamma(t\bar u)-F(\gamma(t\bar u))\right) \]
	is a strictly increasing function, we conclude that
	\[c=g(t_{\bar u})<g(1)=\int_{\mathbb{R}^{N}}[W*F(\gamma(\bar u))]\left(\frac{1}{2}f(\gamma(\bar u))\gamma(\bar u)-F(\gamma(\bar u))\right)=c_\infty,\]
	proving our result. $\hfill\Box$\end{proof}\vspace*{.3cm}

\noindent\textit{Proof of Theorem \ref{t1}}. Let $(u_n)$ be the minimizing sequence given by Lemma \ref{gpm}. It follows from Lemmas \ref{PS} and \ref{ccinfty} that $u_n\to u$ such that $I(u)=c$ and $I'(u)=0$. The positivity of $u$ only reproduces the argument already applied to show the positivity of the solution of the asymptotic problem. 
$\hfill\Box$

\section{Proof of Theorem \ref{t2}}\label{regularity}
We start addressing the regularity of the solution of problem \eqref{P} with some lemmas.
\begin{lemma}\label{hipW} Concerning hypothesis $(W_h)$ we have
	\begin{enumerate}
		\item [($i$)] if $r\in \left(\displaystyle\frac{N}{N(2-\theta)+2\sigma \theta},\frac{2N}{N(2-\theta)+2\sigma \theta}\right]$, there exists $\displaystyle p\in \left[1,\frac{2N}{(N-2\sigma)\theta}\right]$ such that \[|u(0,y)|^\theta\in L^p(\mathbb{R}^{N})\] and
		\[\frac{1}{p}+\frac{1}{r}=1+\frac{N(2-\theta)+2\sigma \theta}{2N}.\]
		Furthermore, $F(u(0,y))\in L^p(\mathbb{R}^{N})$ and
		\[|W_1*F(u(0,y))|=:g\in {L^{2N/[N(2-\theta)+2\sigma \theta]}(\mathbb{R}^{N})}.\]
		\item [($ii$)] if $r'$ denotes the conjugate exponent of $r$ and $r>\displaystyle\frac{2N}{N(2-\theta)+2\sigma \theta}$, then $F(u(0,y))\in L^{r'}(\mathbb{R}^{N})$ and $W_1*F(u(0,y))\in L^\infty(\mathbb{R}^{N})$.
	\end{enumerate}
\end{lemma}

\noindent\begin{proof}($i$) We verify the values of $r$ that satisfy the equality \[\frac{1}{p}+\frac{1}{r}=1+\frac{N(2-\theta)+2\sigma \theta}{2N}.\] Observe that $r\in \left(\frac{N}{N(2-\theta)+2\sigma \theta},\frac{2N}{N(2-\theta)+2\sigma \theta}\right]$ if, and only if, $p\in \left[1,\frac{2N}{(N-2\sigma)\theta}\right)$.\\
	
	As consequence of \eqref{immersion2}, $|u(0,y)|^\theta\in L^p(\mathbb{R}^{N})$ and thus $|u(0,y)|^2\in L^p(\mathbb{R}^{N})$ and \eqref{boundF} yields $F(u(0,y))\in L^p(\mathbb{R}^{N})$. So, $|W_1*F(u(0,y))|=g\in {L^{2N/[N(2-\theta)+2\sigma \theta]}(\mathbb{R}^{N})}$ follows from the Hausdorff-Young inequality.
	
	($ii$) Since $W_1\in L^r(\mathbb{R}^{N})$ for $r=\frac{2N}{N(2-\theta)+2\sigma \theta}$ and $r'=\frac{r}{r-1}=\frac{2N}{(N-2\sigma)\theta}$, applying ($i$) we conclude that $F(u(0,y))\in L^{r'}(\mathbb{R}^{N})$ and $W_1*F(u(0,y))\in L^\infty(\mathbb{R}^{N})$ is consequence of Proposition \ref{HYoung}.
	$\hfill\Box$\end{proof}\vspace*{.3cm}

\begin{corollary}\label{cor}We have
	$|W*F(u(0,y))|\leq C+g$ with $g\in L^{{2N/[N(2-\theta)+2\sigma \theta]}}(\mathbb{R}^{N})$.
\end{corollary}

\noindent\begin{proof}An immediately consequence of Lemma \ref{hipW}, since $W_2\in L^\infty(\mathbb{R}^{N})$.
	$\hfill\Box$\end{proof}\vspace*{.4cm}

Following arguments in \cite{ZelatiNolasco}, we have
\begin{lemma}\label{c1} For all $\theta\in \left(2,\frac{2N}{N-2\sigma}\right)$, we have $|u(0,y)|^{\theta-2}\leq 1+g_2$,	where $g_2\in L^{N/(2\sigma)}(\mathbb{R}^{N})$.
\end{lemma}

\noindent\begin{proof}We have
	\[|u(0,y)|^{\theta-2}=|u(0,y)|^{\theta-2}\chi_{\{|u(0,y)|\leq 1\}}+|u(0,y)|^{\theta-2}\chi_{\{|u(0,y)|>1\}}\leq 1+g_2,\]
	with $g_2=|u(0,y)|^{\theta-2}\chi_{\{|u(0,y)|>1\}}$. If $(\theta-2)N/(2\sigma)\leq 2$, then
	\[\int_{\mathbb{R}^{N}}|u(0,y)|^{\frac{(\theta-2)N}{2\sigma}}\chi_{\{|u(0,y)|>1\}}\leq \int_{\mathbb{R}^{N}}|u(0,y)|^2\chi_{\{|u(0,y)|>1\}}\leq\int_{\mathbb{R}^{N}}|u(0,y)|^2<\infty.\]
	
	When $2<(\theta-2)N/(2\sigma)$, then $(\theta-2)N/(2\sigma)\in \left(2,\frac{2N}{N-2\sigma}\right)$ and, as outcome of \eqref{immersion2}, $|u(0,y)|^{\theta-2}\in L^{N/(2\sigma)}(\mathbb{R}^{N})$.
	$\hfill\Box$\end{proof}

\begin{lemma}\label{c2} For all $\theta\in \left(2,\frac{2N}{N-2\sigma}\right)$ we have $h=g|u(0,y)|^{\theta-2}\in L^{N/(2\sigma)}(\mathbb{R}^{N})$, where $g$ is the function of Lemma \ref{hipW}.
\end{lemma}

\noindent\begin{proof} Since the application of the Hölder inequality yields
\[\int_{\mathbb{R}^{N}}\left(g|u(0,y)|^{\theta-2}\right)^{N/(2\sigma)}\leq \left(\int_{\mathbb{R}^{N}}g^{\alpha N/(2\sigma)}\right)^{\frac{1}{\alpha}}\left(\int_{\mathbb{R}^{N}}\left(|u(0,y)|^{(\theta-2)N/(2\sigma)}\right)^{\alpha'}\right)^{\frac{1}{\alpha'}},\]
if we define $\alpha$ so that $\alpha N/(2\sigma)=2N/[N(2-\theta)+2\sigma\theta]$, then $\alpha'=4\sigma/[(N-2\sigma)(\theta-2)]$ and we have $\alpha'N(\theta-2)/(2\sigma)=2N/[N-2\sigma]$. Since both integrals of the right-hand side of the last inequality are integrable, we are done.
$\hfill\Box$\end{proof}\hspace*{.2cm}

The proof of the next result adapts arguments in \cite{Cabre} and \cite{ZelatiNolasco}. We denote $v_+(z)=\max\{0,v(z)\}$.
\begin{proposition}\label{p1} Let $v\in H^1(\mathbb{R}^{N+1}_+,x^{1-2\sigma})$ be any solution of \eqref{P}. For all $\beta>0$ it holds
	\begin{align*}
	\hspace*{-.25cm}|v^{1+\beta}_+(0,y)|^2_{2^{*}_\sigma}&\leq 2C^2_{2^{*}_\sigma}C_\beta\left[\left(|V|_\infty+CC_1(2+M)\right)\left|v^{1+\beta}_+(0,y)\right|^2_{2}\right.\nonumber\\
	&\qquad\qquad\quad\left.+C_1|g|_{2N/[N(2-\theta)+2\sigma\theta]}\left|v^{1+\beta}_+(0,y)\right|^2_{2^{*}_\sigma(2/\theta)}\right],
	\end{align*}
	where $C_\beta=\max\{m^{-2},\left(1+\frac{\beta}{2}\right)\}$, $C,C_1,\tilde{C}$ and $M=M(\beta)$ are positive constants and $g=|W_1*F(v(0,y)))|$ is the function given by Lemma \ref{hipW}.
\end{proposition}

\noindent\begin{proof}
Choosing $\varphi=\varphi_{\beta,T}=vv^{2\beta}_T$ in \eqref{derivative}, where $v_T=\min\{v_+,T\}$ and $\beta>0$, we have $0\leq \varphi_{\beta,T}\in H^1(\mathbb{R}^{N+1}_+,x^{1-2\sigma})$ and
\begin{multline}\label{varphibetaT}
\iint_{\mathbb{R}^{N+1}_+}\left(\nabla v\cdot\nabla \varphi_{\beta,T}+m^2v\varphi_{\beta,T}\right)x^{1-2\sigma}\\
=-\int_{\mathbb{R}^{N}}V(y)v(0,y)\varphi_{\beta,T}(0,y)+\int_{\mathbb{R}^{N}}\left[W*F(v(0,y))\right]f(v(0,y))\varphi_{\beta,T}(0,y),
\end{multline}
Since $\nabla\varphi_{\beta,T}=v^{2\beta}_T\nabla v+2\beta vv^{2\beta-1}_T\nabla v_T$,
the left-hand side of \eqref{varphibetaT} is given by
\begin{multline}\label{varphibetaTl}
\iint_{\mathbb{R}^{N+1}_+}\left(\nabla v\cdot \left(v^{2\beta}_T\nabla v+2\beta vv^{2\beta-1}_T\nabla v_T\right)+m^2v\left(vv^{2\beta}_T\right)\right)x^{1-2\sigma}\\
=\iint_{\mathbb{R}^{N+1}_+}v^{2\beta}_T\left(|\nabla v|^2+m^2v^2\right)x^{1-2\sigma}+2\beta\iint_{D_T}v^{2\beta}_T|\nabla v|^2\,x^{1-2\sigma},
\end{multline}
where $D_T=\{(x,y)\in (0,\infty)\times \mathbb{R}^{N}\,:\, v_T(x,y)\leq T\}$.

Now we express \eqref{varphibetaTl} in terms of $\|vv^\beta_T\|^2_\sigma$. For this, we note that $\nabla(vv^\beta_T)=v^\beta_T\nabla v+\beta vv^{\beta-1}_T\nabla v_T$. Therefore,
\[\iint_{\mathbb{R}^{N+1}_+}|\nabla (vv^\beta_T)|^2x^{1-2\sigma}=\iint_{\mathbb{R}^{N+1}_+}v^{2\beta}_T|\nabla v|^2\,x^{1-2\sigma}+(2\beta+\beta^2)\iint_{D_T}v^{2\beta}_T|\nabla v|^2x^{1-2\sigma},\]
thus yielding
\begin{align}\label{norm}
\|vv^\beta_T\|^2_\sigma
&=\iint_{\mathbb{R}^{N+1}_+}v^{2\beta}_T|\nabla v|^2x^{1-2\sigma}+(2\beta+\beta^2)\iint_{D_T}v^{2\beta}_T|\nabla v|^2x^{1-2\sigma}\nonumber\\
&\qquad +\iint_{\mathbb{R}^{N+1}_+}(vv^\beta_T)^2x^{1-2\sigma}\nonumber\\
&=\iint_{\mathbb{R}^{N+1}_+}v^{2\beta}_T\left(|\nabla v|^2+|v|^2\right)x^{1-2\sigma}+2\beta\left(1+\frac{\beta}{2}\right)\iint_{D_T}v^{2\beta}_T|\nabla v|^2x^{1-2\sigma}\nonumber\\
&\leq C_\beta\left[\iint_{\mathbb{R}^{N+1}_+}v^{2\beta}_T\left(|\nabla v|^2+m^2|v|^2\right)x^{1-2\sigma}+2\beta\iint_{D_T}v^{2\beta}_T|\nabla v|^2x^{1-2\sigma}\right],
\end{align}
where $C_\beta=\max\left\{m^{-2},\left(1+\frac{\beta}{2}\right)\right\}$. Gathering \eqref{varphibetaT}, \eqref{varphibetaTl} and \eqref{norm}, we obtain
\begin{align}\label{norm=r}
\|vv^\beta_T\|^2_\sigma\leq& C_\beta\left[-\int_{\mathbb{R}^{N}}V(y)v^2(0,y)v^{2\beta}_T(0,y)\right.\nonumber\\
&\qquad+\left.\int_{\mathbb{R}^{N}}\left[W*F(v(0,y))\right]f(v(0,y))v(0,y)v^{2\beta}_T(0,y)\right.
\end{align}

We now start to consider the right-hand side of \eqref{norm=r}. Since $|f(t)|\leq C_1(|t|+|t|^{\theta-1})$, Corollary \ref{cor} shows that it can be written as
\begin{align}\label{rhs}
\leq& C_\beta\left[|V|_\infty\int_{\mathbb{R}^{N}}v^2(0,y)v^{2\beta}_T(0,y)]^{2}+\int_{\mathbb{R}^{N}}(C+g)|f(v(0,y))|\,|v(0,y)|v^{2\beta}_T(0,y)\right]\nonumber\\
\leq&C_\beta\left[|V|_\infty\int_{\mathbb{R}^{N}}v^2(0,y)v^{2\beta}_T(0,y)]^{2}+C\int_{\mathbb{R}^{N}}C_1\left(|v(0,y))|+|v(0,y)|^{\theta-1}\right)|v(0,y)|v^{2\beta}_T(0,y)\right.\nonumber\\
&\qquad+\left.C_1\int_{\mathbb{R}^{N}}g\left(|v(0,y)|+|v(0,y)|^{\theta-1}\right)\,|v(0,y)|v^{2\beta}_T(0,y)\right]\nonumber\\
\leq& C_\beta\left[\left(|V|_\infty+CC_1\right)\int_{\mathbb{R}^{N}}v^2(0,y)v^{2\beta}_T(0,y)+CC_1\int_{\mathbb{R}^{N}}|v^{\theta-2}(0,y)|v^2(0,y)v^{2\beta}_T(0,y)\right.\nonumber\\
&\qquad+\left.C_1\int_{\mathbb{R}^{N}}gv^2(0,y)v^{2\beta}_T(0,y)+C_1\int_{\mathbb{R}^{N}}g|v^{\theta-2}(0,y)|v^2(0,y)v^{2\beta}_T(0,y)\right]
\end{align}

Applying Lemmas \ref{c1} and \ref{c2}, inequality \eqref{rhs} becomes
\begin{align}\label{rhs2}
\leq&C_\beta\left[\left(|V|_\infty+CC_1\right)\int_{\mathbb{R}^{N}}v^2(0,y)v^{2\beta}_T(0,y)+CC_1\int_{\mathbb{R}^{N}}\left(1+g_2\right)v^2(0,y)v^{2\beta}_T(0,y)\right.\nonumber\\
&\qquad+\left.C_1\int_{\mathbb{R}^{N}}gv^2(0,y)v^{2\beta}_T(0,y)+C_1\int_{\mathbb{R}^{N}}hv^2(0,y)v^{2\beta}_T(0,y)\right]\nonumber\\
\leq&C_\beta\left[\left(|V|_\infty+2CC_1\right)\int_{\mathbb{R}^{N}}v^2(0,y)v^{2\beta}_T(0,y)+CC_1\int_{\mathbb{R}^{N}}gv^2(0,y)v^{2\beta}_T(0,y)\right.\nonumber\\
&\left.\qquad+CC_1\int_{\mathbb{R}^{N}}Gv^2(0,y)v^{2\beta}_T(0,y)\right],
\end{align}
where $G=g_2+h\in L^{N/(2\sigma)}(\mathbb{R}^{N})$, admitting that $CC_1\geq C_1$.

Because $|u(0,y)|_{2^*_\sigma}\leq C_{2^{*}_\sigma}\|u\|_\sigma$ for all $u\in H^1(\mathbb{R}^{N+1}_+,x^{1-2\sigma})$, the combination of this immersion with \eqref{norm=r} and \eqref{rhs2} produces
\begin{align}\label{rhs3}|v(0,y)v^{\beta}_T(0,y)|^2_{2^{*}_\sigma}
&\leq C^2_{2^{*}_\sigma} C_\beta\left[\left(|V|_\infty+2CC_1\right)\int_{\mathbb{R}^{N}}v^2(0,y)v_T^{2\beta}(0,y)\right.\nonumber\\
&\quad\left.+CC_1\int_{\mathbb{R}^{N}}gv^2(0,y)v_T^{2\beta}(0,y)+\int_{\mathbb{R}^{N}}Gv^2(0,y)v_T^{2\beta}(0,y)\right].
\end{align}

Let us turn our attention to the last integral in the right-hand side of \eqref{rhs3}. For all $M>0$, define $A_1=\{G\leq M\}$ and $A_2=\{G>M\}$. Then, since Lemmas \ref{c1} and \ref{c2} guarantee that $G=g_2+h\in L^{N/(2\sigma)}(\mathbb{R}^{N})$,
\begin{align*}
\int_{\mathbb{R}^{N}}Gv^2(0,y)v_T^{2\beta}(0,y)&\leq M\int_{A_1}v^2(0,y)v_T^{2\beta}(0,y)\\
&\qquad+\left(\int_{A_2}G^{N/(2\sigma)}\right)^{\frac{2\sigma}{N}}\left(\int_{A_2}\left(v^2(0,y)v_T^{2\beta}(0,y)\right)^{\frac{N}{N-2\sigma}}\right)^{\frac{N-2\sigma}{N}}\nonumber\\
&\leq M\int_{\mathbb{R}^{N}}v^2(0,y)v_T^{2\beta}(0,y)\\
&\qquad+\epsilon(M)\left(\int_{\mathbb{R}^{N}}\left(v(0,y)v^\beta_T(0,y)\right)^{2^{*}_\sigma}\right)^{\frac{N-2\sigma}{N}},
\end{align*}
and $\epsilon(M)=\left(\int_{A_2}G^{N/(2\sigma)}\right)^{2\sigma/N}\to 0$ when $M\to\infty$.

If $M$ is taken so that $\epsilon(M)C^2_{2^{*}_\sigma}C_\beta CC_1<1/2$, we have
\begin{align}\label{rhs4}
|v(0,y)v^{\beta}_T(0,y)|^2_{2^{*}_\sigma}&\leq 2C^2_{2^{*}_\sigma}C_\beta
\left[\left(|V|_\infty+CC_1(2+M)\right)\int_{\mathbb{R}^{N}}v^2(0,y)v_T^{2\beta}(0,y)\right.\nonumber\\
&\qquad\qquad\quad\left.+CC_1\int_{\mathbb{R}^{N}}gv^2(0,y)v_T^{2\beta}(0,y)\right].
\end{align}

The Hölder inequality guarantees that
\begin{align*}
\int_{\mathbb{R}^{N}}gv^2(0,y)v_T^{2\beta}(0,y)\leq |g|_{2N/[N(2-\theta)+2\sigma\theta]}\left(\int_{\mathbb{R}^{N}}\left(v^2(0,y)v_T^{2\beta}(0,y)\right)^{\alpha'}\right)^{1/\alpha'},
\end{align*}
where
\[\alpha'=\frac{\frac{2N}{N(2-\theta)+2\sigma\theta}}{\frac{2N}{N(2-\theta)+2\sigma\theta}-1}=\frac{2N}{(N-2\sigma)\theta}=\frac{2^{*}_\sigma}{\theta}.\]
Thus,
\begin{align*}
\int_{\mathbb{R}^{N}}gv^2(0,y)v_T^{2\beta}(0,y)\leq |g|_{2N/[N(2-\theta)+2\sigma\theta]}\,|v(0,y)v_T^{\beta}(0,y)|^2_{2^{*}_\sigma(2/\theta)}
\end{align*}
and substitution on the right-hand side of \eqref{rhs4} yields
\begin{align}\label{rhs5}
\hspace*{-.25cm}|v(0,y)v_T^{\beta}(0,y)|^2_{2^{*}_\sigma}&\leq 2C^2_{2^{*}_\sigma}C_\beta\left[\left(|V|_\infty+CC_1(2+M)\right)|v(0,y)v_T^{\beta}(0,y)|^2_{2}\right.\nonumber\\
&\qquad\qquad\quad\left.+C_1|g|_{2N/[N(2-\theta)+2\sigma\theta]}\,|v(0,y)v_T^{\beta}(0,y)|^2_{2^{*}_\sigma(2/\theta)}\right],
\end{align}

Since $v(0,y)v_T^{\beta}(0,y)\to v^{1+\beta}_+$, it follows from \eqref{rhs5} that
\begin{align*}
\hspace*{-.25cm}\left|v^{1+\beta}_+(0,y)\right|^2_{2^{*}_\sigma}&\leq 2C^2_{2^{*}_\sigma}C_\beta\left[\left(|V|_\infty+CC_1(2+M)\right)\left|v^{1+\beta}_+(0,y)\right|^2_{2}\right.\\
&\qquad\qquad\quad\left.+C_1|g|_{2N/[N(2-\theta)+2\sigma\theta]}\,\left|v^{1+\beta}_+(0,y)\right|^2_{2^{*}_\sigma(2/\theta)}\right],
\end{align*}
and we are done. (Observe, however, that $M$ depends on $\beta$.) $\hfill\Box$\end{proof}\vspace*{.2cm}

\begin{proposition}\label{p2}
	For all $p\in [2,\infty)$ we have $v(0,\cdot)\in L^p(\mathbb{R}^{N})$.
\end{proposition}

\noindent\begin{proof} Since $\frac{2N}{N-2\sigma}\frac{2}{\theta}\leq 2$ never occurs, we have $2<\frac{2^{*}_\sigma 2}{\theta}=\frac{2N}{N-2\sigma}\frac{2}{\theta}<2^{*}_\sigma$.

According to the Proposition \ref{p1}, we have
\begin{align}\label{bs1}
\left|v^{1+\beta}_+(0,y)\right|^2_{2^{*}_\sigma}&\leq \left[D_1\left|v^{1+\beta}_+(0,y)\right|^2_{2}+E_1\left|v^{1+\beta}_+(0,y)\right|^2_{2^{*}_\sigma(2/\theta)}\right],
\end{align}
where $D_1$ and $E_1$ are positive constants.

Choosing $\beta_1+1:=(\theta/2)>1$, it follows from \eqref{immersion1aa} that
\[\left|v^{1+\beta}_+(0,y)\right|^2_{2^{*}_\sigma(2/\theta)}=\left|v_+(0,y)\right|^{\theta}_{\frac{2N}{N-2\sigma}}<\infty,\]
from what follows that the right-hand side of \eqref{bs1} is finite. We conclude that $v^{1+\beta}_+(0,\cdot)\in {L^{\frac{2N}{N-2\sigma}}}(\mathbb{R}^{N})<\infty$. Now, we choose $\beta_2$ so that $\beta_2+1=(\theta/2)^2$ and conclude that
\[v_+(0,\cdot)\in L^{\frac{2N}{N-2\sigma}\frac{\theta^2}{2^2}}(\mathbb{R}^{N}).\]

After $k$ iterations we obtain that
\[v_+(0,\cdot)\in L^{\frac{2N}{N-1}\frac{\theta^k}{2^k}}(\mathbb{R}^{N}),\]
from what follows that $v_+(0,\cdot)\in L^p(\mathbb{R}^{N})$ for all $p\in [2,\infty)$. Since the same arguments are valid for $v_-$, we have $v(0,\cdot)\in L^p(\mathbb{R}^{N})$ for all $p\in [2,\infty)$.
$\hfill\Box$\end{proof}\vspace*{.2cm}

By simply adapting the proof given in \cite{ZelatiNolasco}, we present, for the convenience of the reader, the demonstration of our next result, which applies Moser's iteration technique:
\begin{proposition}
	Let $v\in H^1(\mathbb{R}^{N+1}_+,x^{1-2\sigma})$ be a weak solution of \eqref{P}. Then $v(0,\cdot)\in L^p(\mathbb{R}^{N})$ for all $p\in [2,\infty]$ and $v\in L^\infty(\mathbb{R}^{N+1}_+)$.
\end{proposition}
\noindent\begin{proof} We recall equation \eqref{norm=r}:
\begin{align*}
\|vv^\beta_T\|^2_\sigma\leq& C_\beta\left[-\int_{\mathbb{R}^{N}}V(y)v^2(0,y)v^{2\beta}_T(0,y)\right.\nonumber\\
&\qquad+\left.\int_{\mathbb{R}^{N}}\left[W*F(v(0,y))\right]f(v(0,y))v(0,y)v^{2\beta}_T(0,y)\right.
\end{align*}
where $C_\beta=\max\{m^{-2},(1+\beta^2)\}$.

It follows that  $W*F(v(0,y))\in L^\infty(\mathbb{R}^{N})$, since $v(0,\cdot)\in L^p(\mathbb{R}^{N})$ for all $p\geq 2$, by Proposition \ref{p2}. We also know that $|f(t)|\leq C_1(|t|+|t|^{\theta-1})$ and $V$ is bounded. Therefore, if $C=\max\{|V|_\infty, C_1|W*F(v(0,y))|_\infty\}$, we have
\begin{align*}
\|vv^\beta_T\|^2_\sigma&\leq C_\beta C\left[\int_{\mathbb{R}^{N}}v^2(0,y)v^{2\beta}_T(0,y)\right.\\
&\qquad\left.+\int_{\mathbb{R}^{N}}\left(|v(0,y))|+|v(0,y)|^{\theta-1}\right)v(0,y)v^{2\beta}_T(0,y)\right]\\
&\leq C_\beta\left[2 C\int_{\mathbb{R}^{N}}v^2(0,y)v_T^{2\beta}(0,y)+ C\int_{\mathbb{R}^{N}}|v(0,y)|^{\theta-2}v^2(0,y)v^{2\beta}_T(0,y)\right].
\end{align*}
Since $|v(0,y)|^{\theta-2}=|v(0,y)|^{\theta-2}\chi_{\{|v(0,y)\leq 1\}}+|v(0,y)|^{\theta-2}\chi_{\{|v(0,y)> 1\}}$, it follows from Proposition \ref{p2} that \[|v(0,y)|^{\theta-2}\chi_{\{|v(0,y)> 1\}}=:g_3\in L^{2N}(\mathbb{R}^{N}),\]
what allow us to conclude that
\begin{align*}2 Cv^2(0,y)v_T^{2\beta}(0,y)+ C|v(0,y)|^{\theta-2}v^2(0,y)v_T^{2\beta}(0,y)\leq (C_3+g_3)v^2(0,y)v_T^{2\beta}(0,y)
\end{align*}
for a positive constant $C_3$ and a positive function $g_3\in L^{2N}(\mathbb{R}^{N})$ that depends neither on $T$ nor on $\beta$.

Therefore,
\begin{align*}
\|vv^\beta_T\|^2_\sigma&\leq \int_{\mathbb{R}^{N}}(C_3+g_3)v^2(0,y)v^{2\beta}_T(0,y).
\end{align*}
and, when $T\to\infty$, by applying Fatou's lemma and the Dominated Convergence Theorem  we obtain
\begin{align*}
\|v^{1+\beta}_+\|^2_\sigma&\leq C_\beta \int_{\mathbb{R}^{N}}(C_3+g_3)\left(v^{1+\beta}_+(0,y)\right)^2.
\end{align*}
Since
\begin{align*}
\int_{\mathbb{R}^{N}}g_3\left(v^{1+\beta}_+(0,y)\right)^2&\leq|g_3|_{2N}\,\left|v^{1+\beta}_+(0,\cdot)\right|_2\, \left|v^{1+\beta}_+(0,\cdot)\right|_{2^{*}_\sigma}\\
&\leq |g_3|_{2N}\left(\lambda\left|v^{1+\beta}_+(0,\cdot)\right|^2_2+\frac{1}{\lambda}\left|v^{1+\beta}_+(0,\cdot)\right|^2_{2^{*}_\sigma}\right),
\end{align*}
we conclude that
\begin{align}
\left|v^{1+\beta}_+(0,\cdot)\right|^2_{2^{*}_{\sigma}}&\leq C^2_{2^{*}_\sigma} \left\|v^{1+\beta}_+\right\|^2_\sigma\label{mest}\\
&\leq C^2_{2^{*}_\sigma}C_\beta\left(C_3+\lambda\,|g|_{2N}\right)\left|v^{1+\beta}_+(0,\cdot)\right|^2_2+\frac{C^2_{2^{*}_\sigma}C_\beta\,|g_3|_{2N}}{\lambda}\left|v^{1+\beta}_+(0,\cdot)\right|^2_{2^{*}_\sigma}\nonumber
\end{align}
and, by taking $\lambda>0$ so that
\[\frac{C^2_{2^{*}_\sigma}C_\beta\,|g_3|_{2N}}{\lambda}<\frac{1}{2},\]
we obtain
\begin{align}
\left|v^{1+\beta}_+(0,\cdot)\right|^2_{2^{*}_{\sigma}}&\leq C_\beta\left(2C^2_{2^{*}_\sigma}C_3+2C^2_{2^{*}_\sigma}\lambda\,|g_3|_{2N}\right)\left|v^{1+\beta}_+(0,\cdot)\right|^2_2\nonumber\\
&\leq C_4C_{\beta}\left|v^{1+\beta}_+(0,\cdot)\right|^2_2.\label{fest}
\end{align}
Since
\[C_4C_\beta=C_4(m^{-2}+1+\beta)\leq M^2e^{2\sqrt{1+\beta}}\]
for a positive constant $M$, it follows from \eqref{fest} that
\begin{align*}
\left|v^{1+\beta}_+(0,\cdot)\right|_{2^{*}_\sigma(1+\beta)}&\leq M^{1/(1+\beta)}e^{1/\sqrt{1+\beta}}|v_+(0,y)|_{2(1+\beta)}
\end{align*}
We now apply an iteration argument, taking $2(1+\beta_{n+1})=2^{*}_\sigma\beta_n$ and starting with $\beta_0=0$. This produces
\[\left|v^{1+\beta}_+(0,\cdot)\right|_{2^{*}_\sigma(1+\beta_n)}\leq M^{1/(1+\beta_n)}e^{1/\sqrt{1+\beta_n}}\left|v^{1+\beta}_+(0,\cdot)\right|_{2(1+\beta_n)}.\]
Because $(1+\beta_n)=\left(\frac{2^{*}_\sigma}{2}\right)^n=\left(\frac{N}{N-2\sigma}\right)^n$,
we have
\[\sum_{i=0}^\infty \frac{1}{1+\beta_n}<\infty\qquad\textrm{and}\qquad \sum_{i=0}^\infty\frac{1}{\sqrt{1+\beta_n}}<\infty.\]

Thus,
\[\left|v^{1+\beta}_+(0,\cdot)\right|_\infty=\lim_{n\to\infty}\left|v^{1+\beta}_+(0,\cdot)\right|_{2^{*}_\sigma(1+\beta_n)}<\infty,\]
from what follows $\left|v^{1+\beta}_+(0,\cdot)\right|_p<\infty$ for all $p\in [2,\infty]$. Since the same argument applies to $v_-(0,\cdot)$,we have that $v(0,\cdot)\in L^p(\mathbb{R}^{N})$ for all $p\in [2,\infty]$.

It follows from \eqref{mest} that
\begin{align*}\left\|v^{1+\beta}_+\right\|^2_\sigma
&\leq C_\beta\left(C_3+\lambda\,|g|_{2N}\right)\left|v^{1+\beta}_+(0,\cdot)\right|^2_2+\frac{C_\beta\,|g_3|_{2N}}{\lambda}\left|v^{1+\beta}_+(0,\cdot)\right|^2_{2^{*}_\sigma}
\end{align*}
By taking $\lambda=1$ and $\left|v^{1+\beta}_+(0,\cdot)\right|_p<C_5$ for all $p$, we obtain for any $\beta>0$,
\begin{align}\label{final0}
\left\|v^{1+\beta}_+\right\|^2_\sigma
&\leq C_\beta\left(C_3+|g|_{2N}\right)C^{2}_5+C_\beta\,|g_3|_{2N}C^{2}_5.
\end{align}
But \[|v_+|^{1+\beta}_{\sigma,2(1+\beta)}=|v_+^{1+\beta}|_{\sigma,2}\leq \|v^{1+\beta}_+\|_\sigma\] and for a positive constant $\tilde{c}$ results from \eqref{final0}  that
\[|v_+|^{2(1+\beta)}_{\sigma,2(1+\beta)}\leq \tilde{c}C_\beta C^{2(1+\beta)}_5.\]
Thus,
\[|v_+|_{\sigma,2(1+\beta)}\leq \tilde{c}^{1/2(1+\beta)}C_\beta^{1/2(1+\beta)}C_5 \]
and the right-hand side of the last inequality is uniformly bounded for all $\beta>0$. Taking into account the next result, we have $v\in L^\infty(\mathbb{R}^{N+1}_+)$. 

To conclude the proof of Theorem \ref{t2} it suffices to apply Theorems 2.3.12 and 2.3.15 in Fabes, Kenig and Serapioni \cite{Fabes}, see also \cite[Theorem 3.3]{CabreSire}.
$\hfill\Box$\end{proof}

\begin{lemma}
	Suppose that there exists a constant $C$ such that $|v|_{\sigma,q}\leq C$ for all $q\in [2,\infty)$. Then $v\in L^\infty(\mathbb{R}^{N+1}_+)$.
\end{lemma}
\begin{proof}
Let $A=\{z=(x,y)\in \mathbb{R}^{N+1}_+\,:\, v(z)>k\}$, where $k\in\mathbb{N}$. Of course we have
\[\iint_{\mathbb{R}^{N+1}_+}|v(z)|^px^{1-2\sigma}\geq \iint_A |v(z)|^px^{1-2\sigma}\geq k^{p-2}\iint_A |v(z)|^2x^{1-2\sigma}.\]
Therefore,
\[C^p\geq k^{p-2}\iint_A |v(z)|^2x^{1-2\sigma}\quad\Rightarrow\quad \left(\frac{C}{k}\right)^{p-2}\geq \frac{1}{C^2}\iint_A |v(z)|^2x^{1-2\sigma},\ \ \forall\ p\geq 2.\]
Taking $k>C$ and making $p\to \infty$, we conclude that
\[\iint_A |v(z)|^px^{1-2\sigma}=0\quad\Rightarrow\quad |A|=0\quad\Rightarrow\quad v\in L^\infty(\mathbb{R}^{N+1}_+),\]
where $|A|$ denotes de Lebesgue measure of $A$. We are done.
$\hfill\Box$\end{proof}\vspace{.2cm}

\begin{remark}In the appendix we prove a Hopf-type principle for our equation. Considering a classical Harnack principle (see Gilbarg-Trundinger), we conclude that the solution $v$ is positive in $\overline{\mathbb{R}^{N+1}_+}$.
	\end{remark}
\section{Proof of Theorem \ref{t3}}\label{decay}
Since $v$ satisfies
\[\left\{\begin{array}{ll}
\displaystyle\Delta_y u+\frac{1-2\sigma}{x}u_x+u_{xx}-m^2u=0 &\text{in }\ \mathbb{R}^{N+1}_+=(0,\infty)\times\mathbb{R}^N,\\
u(0,y)=h(y)\in L^2(\mathbb{R}^{N}), &y\in \mathbb{R}^{N}=\partial \mathbb{R}^{N+1}_+,
\end{array}\right.\]
by applying the Fourier transform we have
\[\hat{v}(x,\xi)=\hat{h}(\xi)\Phi_\sigma\left(\sqrt{m^2+4\pi^2|\xi|^2}\,x\right),\] 
where $\Phi_{\sigma}$ satisfies \eqref{asymp}. Denoting $\hat{\psi}=\Phi_{\sigma}$, it follows that
\[v(x,k)=\mathcal{F}^{-1}(\hat{h}\hat{\psi})=(h*\psi)(k).\]

Therefore, by Parseval,
\[|v(x,k)|\leq \left(\int_{\mathbb{R}^N}|h|^2\dd\xi\right)^{1/2}\left(\int_{\mathbb{R}^N}|\psi(\xi)|^2\dd\xi\right)^{1/2}=|h|_2\left(\int_{\mathbb{R}^N}|\psi(\xi)|^2\dd\xi\right)^{1/2}.\]
Since $\Phi_{\sigma}$ is continuous, for $x\geq 1$ it follows from \eqref{asymp} that
\[\left|\Phi_{\sigma}\left(\sqrt{m^2+4\pi^2|\xi|^2}\,x\right)\right|\leq C_1\left(\sqrt{m^2+4\pi^2|\xi|^2}\,x\right)^{\frac{2\sigma-1}{2}}e^{-x\sqrt{m^2+4\pi^2|\xi|^2}},\]
thus implying that
\begin{equation}\label{boundv}|v(x,k)|\leq |h|_2C_1\left(\int_{\mathbb{R}^N}\left(\sqrt{m^2+4\pi^2|\xi|^2}\,x\right)^{2\sigma-1}e^{-2x\sqrt{m^2+4\pi^2|\xi|^2}}\dd\xi\right)^{1/2}.\end{equation}

We now estimate the integral in the right-hand side of \eqref{boundv}.

\noindent \textbf{Case 1}: $0\leq\sigma\leq 1/2$.

Since $\sqrt{m^2+4\pi^2|\xi|^2}\,x\geq mx$ and $2\sigma-1\leq 0$, we have 
\[\left(\sqrt{m^2+4\pi^2|\xi|^2}\,x\right)^{2\sigma-1}\leq (mx)^{2\sigma-1},\]
thus yielding
\begin{multline}\label{boundv2}\int_{\mathbb{R}^N}\left(\sqrt{m^2+4\pi^2|\xi|^2}\,x\right)^{2\sigma-1}e^{-2x\sqrt{m^2+4\pi^2|\xi|^2}}\dd\xi\\
\leq (mx)^{2\sigma-1}\int_{\mathbb{R}^N}e^{-2x\sqrt{m^2+4\pi^2|\xi|^2}}\dd\xi.\end{multline}

We now consider the integral in the right-hand side of \eqref{boundv2}. 

Take $R>0$ such that $4\pi^2|\xi|^2\geq 3m^2$ for all $\xi\in B^c_R(0)=\mathbb{R}^N\setminus B_R(0)$. Then 
\begin{equation}\label{b1}
\sqrt{m^2+4\pi^2|\xi|^2}\,x\geq 2mx.
\end{equation}

Since, if $x\geq 1$ we always have 
\begin{equation}\label{b2}
\sqrt{m^2+4\pi^2|\xi|^2}\,x\geq 2x\pi|\xi|,
\end{equation} 
it follows from \eqref{b1} and \eqref{b2} that
\begin{equation}\label{b3}
-2x\sqrt{m^2+4\pi^2|\xi|^2}\leq -(2mx+2x\pi|\xi|)
\end{equation}
for all $\xi\in B^c_R(0)$.

Thus, 
\begin{align}\label{b4}\int_{\mathbb{R}^N}e^{-2x\sqrt{m^2+4\pi^2|\xi|^2}}\dd\xi&=\int_{B_R(0)}e^{-2x\sqrt{m^2+4\pi^2|\xi|^2}}\dd\xi+\int_{B^c_R(0)}e^{-2x\sqrt{m^2+4\pi^2|\xi|^2}}\dd\xi\nonumber\\
&\leq \int_{B_R(0)}e^{-2xm}\dd\xi+\int_{B^c_R(0)}e^{-(2mx+2\pi x|\xi|)}\dd\xi\nonumber\\
&\leq e^{-2mx}|B_R(0)|+e^{-2mx}\int_{B^c_R(0)}e^{-2\pi|\xi|}\dd\xi\nonumber\\
&\leq C_2e^{-2mx}.
\end{align}

Gathering \eqref{boundv},\eqref{boundv2} and \eqref{b4}, we conclude that 
\begin{align*}|u(x,k)|&\leq C_1|h|_2\left((mx)^{2\sigma-1}C_2e^{-2mx}\right)^{1/2}\\
&\leq C|h|_2x^{(2\sigma-1)/2}e^{-mx},
\end{align*}
concluding the proof of our claim in the case $0\leq\sigma\leq 1/2$.

\noindent \textbf{Case 2}: $1/2<\sigma<1$.
\begin{multline*}
\int_{\mathbb{R}^N}\left(\sqrt{m^2+4\pi^2|\xi|^2}\,x\right)^{2\sigma-1}e^{-2x\sqrt{m^2+4\pi^2|\xi|^2}}\dd\xi
\end{multline*}
\begin{align*}
&=x^{2\sigma-1}\int_{\mathbb{R}^N}\left(\sqrt{m^2+4\pi^2|\xi|^2}\right)^{2\sigma-1}e^{-2x\sqrt{m^2+4\pi^2|\xi|^2}}\dd\xi\\
&\leq x^{2\sigma-1}\int_{\mathbb{R}^N}\left(m^{2\sigma-1}+(2\pi|\xi|)^{2\sigma-1}\right)e^{-2x\sqrt{m^2+4\pi^2|\xi|^2}}\dd\xi\\
&=(mx)^{2\sigma-1}\int_{\mathbb{R}^N}e^{-2x\sqrt{m^2+4\pi^2|\xi|^2}}\dd\xi+x^{2\sigma-1}\int_{\mathbb{R}^N}(2\pi|\xi|)^{2\sigma-1}e^{-2x\sqrt{m^2+4\pi^2|\xi|^2}}\dd\xi
\end{align*}

We now estimate the second integral in the last inequality. By simply adapting the reasoning in Case 1, take $R>0$ such that $4\pi^2|\xi|^2\geq 3m^2$ for all $\xi\in B^c_R(0)$ and obtain inequality \eqref{b3}. Thus,
\begin{multline*}
\int_{\mathbb{R}^N}|\xi|^{2\sigma-1}e^{-2x\sqrt{m^2+4\pi^2|\xi|^2}}\dd\xi
\end{multline*}
\begin{align}\label{b5}
&=\int_{B_R(0)}|\xi|^{2\sigma-1}e^{-2x\sqrt{m^2+4\pi^2|\xi|^2}}\dd\xi+\int_{B^c_R(0)}|\xi|^{2\sigma-1}e^{-2x\sqrt{m^2+4\pi^2|\xi|^2}}\dd\xi\nonumber\\
&\leq e^{-2mx}R^{2\sigma-1}|B_R(0)|+e^{-2mx}\int_{\mathbb{R}^N}|\xi|^{2\sigma-1}e^{-2\pi|\xi|}\dd\xi\nonumber\\
&\leq C_3e^{-2mx}.
\end{align}
It follows easily the Claim also in the case $1/2<\sigma<1$.
$\hfill\Box$

\section{Appendix}
We start this section by stating a weak maximum principle. Its proof is standard, by taking a nonnegative test function $\varphi$ and then $\varphi=u^-=\max\{-u,0\}$.
\begin{lemma}[Weak Maximum Principle] If $u\in H^1(\mathbb{R}^{N+1}_+)$ satisfies
\[\left\{\begin{array}{rcll}
-\diver\left(x^{1-2\sigma}\nabla u\right)+m^2x^{1-2\sigma}u&\geq&0 &\text{in }\ B^+_R\\
\displaystyle\lim_{x\to 0^+} \left(-x^{1-2\sigma}\frac{\partial u}{\partial x}\right)&\geq&0 &\text{on }\ \Gamma^0_R\\
u&\geq& 0 &\text{on }\ \Gamma^+_R\end{array},\right.\]
then $u\geq 0$ in $B^+_R$.
\end{lemma}
The strong maximum principle is also valid, see \cite{Fabes}. 

We now prove a Hopf principle, adapting the proof given by X. Cabré and Y. Sire \cite[Proposition 4.11]{CabreSire}. We keep up with the notation introduced in that paper: $\Gamma^0_R$ denotes the ball of center $0$ and radius $R$ in $\mathbb{R}^N$.

\begin{proposition}[Hopf Principle] Consider the cylinder $C_{R,1}=\Gamma^0_R\times (0,1)\subset \mathbb{R}^{N+1}_+$. Suppose that $u\in C(\overline{C_{R,1}})\cap H^1(C_{R,1},x^{1-2\sigma})$ satisfy
\[\left\{\begin{array}{rcll}
\diver\left(x^{1-2\sigma}\nabla u\right)-m^2x^{1-2\sigma}u&\geq&0 &\text{in }\ C_{R,1}\\
u&>& 0 &\text{in }\ C_{R,1}\\
u(0,0)&=&0\end{array},\right.\]
Then, 
\[\lim_{x\to 0^+} \left(-x^{1-2\sigma}\frac{u(x,0)}{x}\right)\leq0.\]
In addition, if $x^{1-2\sigma}u_x\in C(\overline{C_{R,1}})$, then 
\[\lim_{x\to 0^+}\left(-x^{1-2\sigma}\frac{\partial u}{\partial x}(x,0)\right)<0.\]
\end{proposition}
\begin{proof} (Sketch) Let us denote $L_au=\diver(x^a\nabla u)-m^2x^au$ and $a=1-2\sigma$. We consider $\omega_A\colon C_{R,1}\to\mathbb{R}$ given by
\[\omega_A(x,y)=x^{-a}\left(xe^{Ax}\phi(y)\right)=x^{-a+1}e^{Ax}\phi(y),\]
where the constant $A$ will be chosen later on and $\phi=\phi(y)$ is the first positive eigenfunction of $-\Delta_y$ in $\Gamma^0_{R/2}$ with Dirichlet boundary conditions, with $|\phi|_\infty=1$ and $\lambda_1$ as the first eigenvalue.

Since
\begin{align*}
\diver(x^a\nabla \omega_a)-m^2x^a\omega_A&=x^a\nabla_y\omega_A+x^a\frac{\partial^2\omega_a}{\partial x^2}+ax^{a-1}\frac{\partial\omega_A}{\partial x}-m^2x^a\omega_A\\
&=\left\{-\lambda_1x+\left[2A(-a+1)+Aa\right]+A^2x-m^2x\right\}e^{Ax}\phi(y)\\
&=\left[A(2-a)+x\left(A^2-m^2-\lambda_1\right)\right]e^{Ax}\varphi(y),
\end{align*}
we have 
\[L_a\omega_A=\diver(x^a\nabla\omega_A)-m^2x^a\omega_A\geq 0\qquad\text{in }\ C_{R/2,1}\] 
by taking $A$ large enough.

Therefore, for $\epsilon>0$,
\[L_a(u-\omega_A)\leq 0\quad\text{in }\ C_{R/2,1}\]
and $u-\epsilon\omega_a=u\geq 0$ on $\partial \Gamma^0_{R/2}\times[0,1)$. 
Taking $\epsilon>0$ small enough, we have
\[u\geq \epsilon \omega_A\quad\text{on }\ \Gamma^0_{R/2}\times \{x=1/2\}.\]
since $u$ is continuous and positive on the closure of this set. Because $\omega_A=0$ on $\Gamma^0_R\times\{x=0\}$, we have
\begin{align*}
L_a(u-\epsilon \omega_A)&\leq 0 \qquad\text{in }\ C_{R/2,1/2}\\
u-\epsilon \omega_A&\geq 0 \qquad\text{on }\ \partial C_{R/2,1/2}.
\end{align*}
 
It follows then from the weak maximum principle that
\[u-\epsilon \omega_A\geq 0\quad\text{in }\ \overline{C_{R/2,1/2}}.\]
Thus, 
\[\lim_{x\to 0^+} \left(-x^{1-2\sigma}\frac{u(x,0)}{x}\right)\leq \epsilon\lim_{x\to 0^+} \left(-x^{1-2\sigma}\frac{w_A(x,0)}{x}\right)=-\epsilon\phi(0)<0,\]
as claimed. 

The last part of the proof follows literally that given in \cite{CabreSire}, simply changing the notation. ``If, additionally, $x^{1-2\sigma}u_x\in C(\overline{C_{R,1}})$, take $x_0\leq 1/2$. Since $(u-\epsilon \omega_A)(\cdot,0)\geq 0$ in $[0,x_0]$ and $(u-\epsilon \omega_A)(0,0)=0$, we have $(u_x-\epsilon (\omega_A)_x)(x_1,0)\geq 0$ for some $x_1\in (0,x_0)$. Repeating the argument for a sequence of $x_0$'s tending to $0$, we conclude that $-x^au_x\leq -\epsilon x^a(\omega_A)_x$ at a sequence of points $(x_j,0)$ with $x_j\downarrow 0$. Since we assume $x^au_x$ continuous up to $\{x=0\}$ and $-\epsilon(x^a(\omega_A)_x)(x_j,0)\to -\epsilon\phi(0)$, we conclude that 
\[\lim_{x\to 0^+}\left(-x^{1-2\sigma}\frac{\partial u}{\partial x}(x,0)\right)<0.\textrm{''}\]
We are done. 
$\hfill\Box$\end{proof}

\end{document}